\documentclass[12pt, final]{amsproc}
\usepackage[english]{babel}
\usepackage{amssymb}
\usepackage{amscd}
\usepackage{mathrsfs}
\usepackage[notcite,notref]{showkeys}
\newcommand{\adef}{\begin{defin}}
\newcommand{\zdef}{\end{defin}}

\newtheorem{defin}{Definition}
\usepackage{amsmath,tikz-cd}
\usepackage[all]{xy}

\newcommand{\ga}{\gamma}

\newcommand{\eps}{\varepsilon}

\newcommand{\n}{\|\cdot\|}

\newcommand{\wh}{\widehat}

\newcommand{\fin}{\operatorname{fin}}

\newcommand{\dist}{\operatorname{dist}}

\newcommand{\Ext}{\operatorname{Ext}}

\newcommand{\supp}{\operatorname{supp}}

\newcommand{\F}{\mathscr F}
\newcommand{\e}{\epsilon}

\newtheorem{thm}{Theorem}[section]
\newtheorem{cor}[thm]{Corollary}
\newtheorem{lem}[thm]{Lemma}
\newtheorem{prop}[thm]{Proposition}

\newtheorem{prob}[thm]{Problem}

\theoremstyle{definition}

\newcommand{\R}{\mathbb R}
\newcommand{\Z}{\mathbb Z}
\newcommand{\N}{\mathbb N}

\newcommand{\T}{\mathbb{T}}
\newcommand{\PO}{\mathrm{PO}}
\newcommand{\PB}{\mathrm{PB}}

\newcommand{\arrow}[1]{\overset{#1}{\longrightarrow}}

\theoremstyle{remark}

\newcommand\restr[2]{{
		\left.\kern-\nulldelimiterspace 
		#1 
		\right|_{#2} 
}}

\newcommand{\To}{\longrightarrow}
\newcommand{\mar}[1]{\marginpar{\tiny{#1}}}

\allowdisplaybreaks

\numberwithin{equation}{section}
\title{When Kalton and Peck met Fourier}
\begin{document}

\author{F\'elix Cabello S\'anchez}
\address{Instituto de Matem\'aticas\\ Universidad de Extremadura\\
Avenida de Elvas\\ 06071-Badajoz\\ Spain} \email{fcabello@unex.es, salgueroalarcon@unex.es}


\author{Alberto Salguero-Alarc\'on}

\thanks{This research has been supported in part by project MTM2016-76958-C2-1-P and Project IB16056 de la Junta de Extremadura. The second author is supported by an FPU grant 2019 by the Spanish Ministry of Education, Science and Universities.}


{\footnotesize{\today}}

\bigskip

\maketitle

\begin{abstract} The paper studies short exact sequences of Banach modules over the convolution algebra $L_1=L_1(G)$, where $G$ is a compact abelian group. The main tool is the notion of a nonlinear $L_1$-centralizer, which in combination with the Fourier transform, is used to produce sequences of $L_1$-modules $0\To L_q\To Z\To L_p\To 0$ that are nontrivial as long as the general theory allows it, namely for $p\in (1,\infty], q\in[1,\infty)$.
Concrete examples are worked in detail for the circle group, with applications to the Hardy classes, and the Cantor group.
\end{abstract}

\section{Introduction}

Most honest Banach spaces carry one or several module structures over a selected group of algebras that reflect the existence of a good supply of natural symmetries.
This applies especially to the Banach spaces of analysis.

Focusing on commutative algebras, ``pointwise'' $L_\infty$-module  
structures are often related to ``unconditionality'', while module structures over convolution algebras $L_1$ appear in connection with translation-invariant properties. The Fourier transform provides a useful mechanism entwinning these properties in function spaces on locally compact abelian groups and their dual groups.

While the homomorphisms of $L_\infty$-modules and $L_1$-modules between function spaces have attracted, under the name of multipliers, a considerable attention in the past (see for instance \cite{mali} and \cite[Section 2.5]{graf}), not much is known on the corresponding extensions. 

An extension of $X$ by $Y$ in the category of Banach modules over $A$ is a short exact sequence of Banach $A$-modules and homomorphisms $0\To Y\To Z\To X\To 0$. These are our main object of study. The first difficulty one faces when dealing with such extensions is the lack of examples, which is due to the fact that even for the simplest choices of $X$ and $Y$ the possible middle objects $Z$ are not ``pre-existing'' $A$-modules.  

Kalton introduced in \cite{k-comm} the notion of a {nonlinear centralizer} for $L_\infty$-modules and effectively used it to construct a wide variety of exact sequences of $L_\infty$-modules  
$0\To L_p\To Z\To L_p\To 0$. In \cite{c-ma} it is shown that every extension of $L_\infty$ modules $0\To L_q\To Z\To L_p\To 0$ is trivial if $p\neq q$ (so that $Z$ is naturally isomorphic to $L_q\oplus L_p$) and that the space of ``self-extensions'' of $L_p$ is basically independent of $p<\infty$.
Roughly speaking, a centralizer between two Banach $A$-modules is a mapping that ``almost intertwines'' their module structures, that is, the sums and the actions of $A$; see Section~\ref{sec:quasi} for the definition.

Replacing $L_\infty$ by the convolution algebra $L_1$ leads to {\em terra incognita}.
This paper makes the first steps in the study of extensions of $L_1$-modules. We provide explicit examples of extensions of Banach $L_1$-modules $0\To L_q\To Z\To L_p\To 0$ that are nontrivial as long as the general theory allows it, namely for $p\in (1,\infty], q\in[1,\infty)$. 
These are constructed and studied by means of $L_1$-centralizers $\mho: L_p\To L_q$ which in turn are obtained from the classical Kalton-Peck maps $\Omega$
for sequence spaces by means of the Fourier tranform.
In this way the (good) behavior of $\Omega$ with respect to the pointwise product ensures that $\mho$ almost commutes with the product of convolution and makes it an $L_1$-centralizer.
Concrete examples and applications are worked in detail for the circle group and the Cantor group.

Most of the paper deals with the case $p=\infty, q=1$ which is, by far, the most difficult and interesting one. Indeed,
a second motivation for our work has been the understanding of certain twisted sums of Banach spaces of the form $0\To L_1(\mu)\To Z\To C(K)\To 0$. All known depend one way or another on the existence of ``twisted Hilbert spaces'', and ours are not exceptions. This is quite enigmatic since the $\mathscr{L}_1$-spaces are the less injective Banach spaces there are, while those of type $\mathscr{L}_\infty$ are the less projective ones. In the end this could be just the tip of the iceberg of a quasilinear version of Grothendieck factorization theorem, as explained in Section~\ref{sec:CT}. 

\section{Preliminaries}

\subsection{Algebraic setup}
Let $A$ be a complex Banach algebra. A quasinormed module over $A$ is a quasinormed space $X$ with a continuous multiplication $A \times X \to X$ which defines a structure of $A$-module in the algebraic sense. If the space $X$ is complete, then it is called a quasi Banach module. A homomorphism between quasi Banach $A$-modules is an operator $T: X\To Y$  such that $T(ax) = aT(x)$ for every $a\in A$ and every $x\in X$. Operators and homomorphisms are always assumed to be continuous (i.e., bounded), otherwise we speak of linear maps and \emph{morphisms}.

An extension of $X$ by $Y$ is a short exact sequence of modules and homomorphisms
\begin{equation}\label{eq:yzx}
\begin{tikzcd}
0 \arrow[r] & Y \arrow[r, "\imath "] & Z \arrow[r, "\pi "] & X \arrow[r] & 0
\end{tikzcd}
\end{equation}
The open mapping theorem guarantees that $\imath$ embeds $Y$ as a closed submodule of $Z$ in such a way that $\pi$ induces an isomorphism between  $Z/\imath[Y]$ and  $X$. We say that $0\To Y\To
Z'\To X\To 0$ is {\em equivalent} to (\ref{eq:yzx}) if there is a homomorphism $u$ making commutative the diagram
$$\xymatrix{
0\ar[r] & Y \ar@{=}[d]  \ar[r] & Z  \ar[d]^u \ar[r] & X\ar[r]\ar@{=}[d] &0\\
0\ar[r] & Y \ar[r] & Z' \ar[r] & X\ar[r] &0
}
$$
The five-lemma and the open mapping theorem guarantee that such a $u$ is an isomorphism, thus revealing that it is a true relation of equivalence. 
 We say that (\ref{eq:yzx}) splits if there is a homomorphism $P:Z\To Y$ such that $P\imath={\bf I}_Y$; equivalently, if there is a homomorphism $S:X\To Z$ such that $\pi S={\bf I}_X$. 
This happens if and only if  (\ref{eq:yzx}) is {\em trivial}, that is,  equivalent to the direct sum extension $0\To Y\To Y\oplus X\To X\To 0$ with the obvious homomorphisms.

We denote by $\Ext_A(X,Y)$ the set of equivalence classes of extensions of $X$ by $Y$. By using pullbacks and pushouts this set can be given a linear structure whose zero element is the class of trivial sequences. Thus,  $\Ext_A(X,Y)=0$ 
 means that each extension of $A$-modules (\ref{eq:yzx}) is trivial.  Taking $A=\mathbb C$ one recovers extensions of quasi Banach spaces. This leads to a delicate point: every extension of modules is also an extension of quasi Banach spaces which could split in the linear category but not necessarily as an extension of modules, because the splitting operators need not to be homomorphisms. This is connected to the (Johnson) amenability of $A$, a major theme in the homology of Banach algebras. Indeed, $A$ is amenable if and only if each extension of (quasi) Banach $A$-modules (\ref{eq:yzx}) in which $Y$ is a dual module which splits as an extension of (quasi) Banach spaces also splits as an extension of $A$-modules; see for instance Runde \cite[Theorem 2.3.21]{runde}.

The (left) $A$-module structure of the dual of a (right) $A$-module $V$ is defined by
$
\langle av^*, v\rangle=  \langle v^*, va\rangle
$
for $a\in A,v^*\in V^*, v\in V$.
An important result of Johnson \cite[Theorem 2.1.10]{runde} states that a locally compact group $G$ is amenable (in the sense that $L_\infty(G)$ has an invariant mean) if and only if the convolution algebra $L_1(G)$ is amenable. Since {\em both} compact {\em and} abelian groups are amenable (the Haar measure is an invariant mean) the convolution algebra $L_1(G)$ is {\em overwhelmingly} amenable when $G$ is a compact abelian group.

And speaking about delicate points let us explain why we risk losing most of our  potential readers by considering quasinorms instead of the more popular norms. Our constructions are invariably Banach spaces and modules whose topology is described by a quasinorm which, in the end, is equivalent to a norm. However, this is not automatic: there exist exact sequences of quasi Banach spaces (\ref{eq:yzx}) in which $X$ and $Y$ are Banach spaces but $Z$ is not. That cannot happen if $X$ is either $L_p(\mu)$ for $1<p\leq\infty$ or a $C(K)$ space because the $\mathscr L_p$-spaces are $\mathscr K$-spaces for those values of $p$. See Section~\ref{sec:quasi}. To be true, only the fact that Hilbert spaces are $\mathscr K\!$-spaces is necessary here because our constructions are pullbacks/pushouts of self-extensions of Hilbert spaces. So there is no cause for concern here.



\subsection{Pushout and pullback}\label{sec:POPB}  In particular, given an extension (\ref{eq:yzx}) and homomorphisms $v:Y\To Y'$ and $v:X'\To X$ there are new extensions forming commutative diagrams
$$
\xymatrixrowsep{0.5pc}
\xymatrix{
0\ar[r] & Y \ar[dd]_v \ar[r]^-\imath &Z \ar[dd]_V \ar[r]^-\pi &X \ar@{=}[dd]  \ar[r] & 0 && 0\ar[r] & Y \ar[r]^-\imath &Z \ar[r]^-\pi &X  \ar[r] & 0\\
&&&&&\text{and}\\
0\ar[r] & Y' \ar[r]^-I &\PO \ar[r]^-\Pi &X  \ar[r] & 0 && 0\ar[r] & Y   \ar@{=}[uu]\ar[r]^-I &\PB \ar[uu]^U \ar[r]^-\Pi &X' \ar[uu]^u  \ar[r] & 0
}
$$
In the diagram on the left, the {\em pushout} space (actually $A$-module) $\PO$ is the quotient
of the direct sum $Y'\oplus Z$ by the (closed) submodule $\{(v(y),-\imath(y)): y\in Y\}$, the maps $I,V$ are the corresponding inclusions followed by the quotient map, and $\Pi$ arises from the factorization of the ``obvious'' homomorphism $(y',z)\longmapsto \pi(z)$. The pushout sequence splits if and only if $v$ extends to $Z$ (there is a homomorphism $\tilde{v}:Z\To Y'$ such that $v=\tilde{v}\imath$).

In the diagram on the right, the {\em pullback} space is $\PB=\{(z,x')\in Z\times X': \pi(z)=u(x')\}$ and the homomorphisms $U, \Pi$ are the restrictions of the projections of  $Z\times X'$ onto $Z$ and $X'$ respectively. The pullback sequence splits if and only if $u$ admits a lifting to $Z$ (there is a homomorphism $\tilde{u}: X'\To X$ such that $u=  \pi\tilde{u})$. 

The {\em explicit} use of these  constructions is marginal in this paper, although they are implicit in most arguments. In any case we remark that is the existence of pullback and pushouts what makes the assingnment $X,Y\rightsquigarrow\Ext_A(X,Y)$ into a functor. 

\subsection{Harmonic setup}

Throughout the paper $G$ denotes a compact abelian group with dual group $\Gamma$. The integral of a function $f:G\To\mathbb C$ with respect to the (normalized) Haar measure of $G$ is denoted by $\int_Gf(x)dx$ and the translation of $f$ along $y\in G$ is defined as $f_y(x)=f(xy^{-1})$. The same definition applies to functions defined on $\Gamma$. As usual, for $0<p<\infty$, we denote by $L_p(G)$ the space of (complex) $p$-integrable functions on $G$ and we identify functions that agree almost everywhere. $L_\infty(G)$ is the space of essentially bounded (measurable) functions with the essential supremum norm. It is clear that $L_\infty(G)$ is a Banach algebra with the pointwise operations and each $L_p(G)$ space is $L_\infty(G)$-module.

We denote by $C(G)$ the algebra of all continuous functions on $G$, again with the sup norm and the pointwise operations; $M(G)$ denotes the space of (bounded, complex, regular, Borel) measures on $G$ normed by the total variation. By the Riesz representation theorem the space of measures on $G$ can be identified with the dual of $C(G)$ through the pairing $\langle \mu,f\rangle=\int_Gf(x)d\mu(x)$. In particular, each $f\in L_1$ defines a measure on $G$ by the formula $A\longmapsto\int_Af(x)dx$ which allows us to consider $L_1(G)$ as a subspace of $M(G)$ when appropriate.

We denote by $P(G)$ the ``polynomials'' on $G$, that is, those functions that arise as linear combinations of characters. While $P(G)$ is dense in $L_p(G)$ for finite $p$, its closure in $L_\infty(G)$ is exactly $C(G)$, by the Stone-Weierstra\ss\ theorem. If $X$ is a space of functions on $G$, we set $X^0=X\cap P(G)$, with the norm inherited from $X$.


The space $M(G)$ is a Banach algebra under the convolution of measures given by $\mu*\nu(A)=\int_G\mu(Ay^{-1})d\nu(y)
$, where $Ay^{-1}=\{xy^{-1}: x\in A\}$. This extends the usual convolution for integrable functions, defined by 
 $(f*g)(x)=\int_G f(xy^{-1})g(y)dy$, in such a way that $L_1(G)$ is an ideal in $M(G)$. 
The spaces $L_p(G)$ for $1\leq p\le\infty$ and $C(G)$ are $L_1(G)$-modules (actually $M(G)$-modules) under convolution; see for instance Rudin \cite[Section 1.1]{rudin}.

We denote by $\F$ the Fourier transform from $G$-objects to $\Gamma$-objects, not the other way around. The largest space where $\F$ is ``naturally'' defined is $M(G)$ and for each measure $\mu$ the Fourier transform $\F(\mu)=\wh \mu$ is the bounded function defined on $\Gamma$ by the formula $\wh\mu(\gamma)=\int_G\overline\gamma(x) d\mu(x)$. The Fourier transform of $f\in L_1(G)$ is defined by $\wh f(\gamma)=\int_Gf(x) \overline\gamma(x) dx$. Quite clearly, $\wh f_y= y\wh f$, with the meaning that $(\wh f_y)(\gamma)=\gamma(y)\wh f(\gamma)$.

\subsection{Quasilinear maps and centralizers}\label{sec:quasi}
Let $X, Y$ be quasinormed spaces. A homogeneous mapping $\Phi:X\To Y$ is said to be quasilinear if $\|\Phi(x+x')-\Phi(x)-\Phi(x')\|\leq Q\big(\|x\|+\|x'\|\big)$ for some constant $Q$ and all $x,x'\in X$. If, in addition, $X,Y$ are modules over a Banach algebra $A$ we say that $\Phi$ is an $A$-centralizer if it is quasilinear and there is a constant $C$ so that $\|\Phi(ax)- a\Phi(x)\|\leq C\|a\|\|x\|$ for all $a\in A$ and all $x\in X$. Quasilinear maps are nothing else than $\mathbb C$-centralizers.

Our interest in quasilinear maps lies on the following construction. If $\Phi:X\To Y$ is a quasilinear map then the functional $\|(y,x)\|_\Phi=\|y-\Phi(x)\|+\|x\|$ is a quasinorm on $Y\times X$ and, if we denote by $Y\oplus_\Phi X$ the resulting quasinormed space, then the sequence
$$
\xymatrix{
0\ar[r] & Y \ar[r]^-\imath & Y\oplus_\Phi X \ar[r]^-\pi &X  \ar[r] & 0
}
$$
where $\imath(y)=(y,0)$ and $\pi(y,x)=x$ is exact. What is more: $\imath$ preserves the involved quasinorms and $\pi$ maps the unit ball of $ Y\oplus_\Phi X$
onto that of $X$. Now, we have:

\begin{lem}\label{lem:modiffcen}
A quasilinear map $\Phi: X\To Y$ is an $A$-centralizer if and only if the outer product $a(y,x)=(ay,ax)$ makes $Y\oplus_\Phi X$ into a quasinormed module.
\end{lem}

\begin{proof}
If $\Phi$ is a centralizer, then
$$
\|(ay,ax)\|_\Phi= \|ay-\Phi(ax)\|+\|ax\| \leq  \|ay-a\Phi(x)+ a\Phi(x)-\Phi(ax)\|+\|ax\|\leq M \|a\|\|(y,x)\|_\Phi.
$$
For the converse, if $
\|(ay,ax)\|_\Phi\leq M
\|a\|\|(y,x)\|_\Phi$ for some $M$ and all $(y,x)\in Y\oplus_\Phi X$, then  letting $y=\Phi(x)$ one obtains $\|a\Phi(x)-\Phi(ax)\|+\|ax\|\leq M\|a\|\|x\|$, and so $\Phi$ is a centralizer.
\end{proof}

In practice one cannot expect to explicitly define centralizers on a quasi Banach module, at least if we insist in that they take values in $Y$. So assume $X, Y$ are quasi Banach modules and that $X_0$ is a dense submodule of $X$. Let $\Phi:X_0\To Y$ be a centralizer and construct the ``twisted sum'' $Y\oplus_\Phi X_0$ as before. Let $Z(\Phi)$ be the completion of  $Y\oplus_\Phi X_0$ (which is a quasi Banach $A$-module) and observe that the universal property of the completion provides a commutative diagram
$$
\xymatrix{
0\ar[r] & Y \ar@{=}[d] \ar[r]^-\imath & Y\oplus_\Phi X_0  \ar[d]^{\text{inclusion}}\ar[r]^-\pi &X_0  \ar[r]\ar[d]^{\text{inclusion}} & 0\\
0\ar[r] & Y \ar[r]^-\jmath & Z(\Phi) \ar[r]^-\varpi &X  \ar[r] & 0
}
$$
in which the lower row is an extension of $X$ by $Y$ called, with good reason, the extension generated by $\Phi$. These extensions admit a very simple analysis: $\Phi$ generates a trivial extension if and only if it there exists a (not necessarily continuous) morphism $\phi:X_0\To Y$ such that $\|\Phi(x)-\phi(x)\|\leq K\|x\|$ for some $K$ and all $x\in X_0$. In this case we say that $\Phi$ is a trivial centralizer. Also $\Phi, \Psi:X_0\To Y$ generate equivalent extensions if and only if $\Phi-\Psi$ is trivial.

A quasilinear map $\Phi: X \To Y$ acting between Banach spaces gives rise to a Banach space if and only if it obeys an estimate of the form $\|\Phi\big(\sum_{i\leq n}x_i\big)-\sum_{i\leq n}\Phi(x_i)\|\leq M\sum_{i\leq n}\|x_i\|$ for some $M$, all $n\in \N$ and all $x_1,\dots,x_n\in X$. Each Banach  $\mathscr K\!$-space $X$ comes with a constant $K=K(X)$ so that every quasilinear map $\Phi$ defined on a dense subspace of $X$ satisfies the above estimate with $M=K(X)Q(\Phi)$.
One has $K(X)\leq 37$ when $X$ is a Hilbert space and $K(X)\leq 200\lambda$ when $X$ is an $\mathscr L_{\infty\lambda}$-space. In particular $K(C(S))\leq 200$ for each compact space $S$; see \cite{k,k-r}.

\subsection{Existence of $L_1$-centralizers}
We now particularize to modules on the convolution algebra $L_1(G)$. We denote by $*$ the action on any $L_1(G)$-module, even if the action might have nothing to do with convolution. The following simple remark justifies our approach in the search for extensions of $L_1$-modules:

\begin{lem}\label{lem:all}
Let $1\leq p<\infty$. 
Every extension of $L_p$  by an arbitrary $L_1$-module $Y$ is equivalent to one induced by an $L_1$-centralizer $\Phi:L_p^0 \To Y$. Every extension of $C(G)$ by $Y$ is equivalent to one induced by an $L_1$-centralizer $\Phi:L_\infty^0 \To Y$.
\end{lem}

\begin{proof}
The key point is that $P(G)$ is a projective $L_1$-module in the purely algebraic sense. We do the proof for the $L_p$ case, but it is analogous for $C(G)$. Consider an extension of $L_1$-modules
$$
\begin{tikzcd}
0 \arrow[r] & Y \arrow[r,"\jmath"] & Z \arrow[r, "\varpi "] & L_p \arrow[r] & 0
\end{tikzcd}
$$
in which we assume that $Y=\ker\varpi$ and $\jmath$ is the inclusion map. Let $B:L_p\To Z$ be a bounded homogeneous section of $\varpi$, which exists by the open mapping theorem. Given $\gamma\in\Gamma$ we consider it as a polynomial and we set $z_\gamma=\gamma*B(\gamma)$; note that $\gamma*z_\gamma=z_\gamma$ and that  $\varpi(z_\gamma)= \varpi(\gamma*B(\gamma))=\gamma*\varpi(B(\gamma))=\gamma*\gamma=\wh\gamma(\gamma)\gamma=\gamma$. Moreover, the map $L: P(G)\To Z$ sending each polynomial $f=\sum_\gamma c_\gamma\gamma$ to $\sum_\gamma c_\gamma z_\gamma$ is an $L_1$-morphism since for each $a\in L_1$ and every polynomial $f=\sum_\gamma c_\gamma\gamma$ one has
\begin{align*}
L(a*f)&=L\Big( \sum_\gamma \wh a(\gamma) c_\gamma \gamma\Big) = \sum_\gamma \wh a(\gamma) c_\gamma z_\gamma,\\
a*L(f)&=a*\sum_\gamma c_\gamma z_\gamma= \sum_\gamma c_\gamma (a*z_\gamma)
= \sum_\gamma c_\gamma (a*\gamma)z_\gamma=  \sum_\gamma c_\ga \wh a(\gamma) z_\gamma.
\end{align*}
Now, the difference $\Phi=B-L$ takes values in $Y=\ker\varpi$ since $\varpi(B(f)-L(f))=f-f=0$ for every $f\in P(G)$ and $\Phi: L_p^0\To Y$ is an $L_1$-centralizer:
\begin{align*}
&\|\Phi(f+g)-\Phi f-\Phi g\|=\|B(f+g)-B(f)-B(g)\| 
\leq 2\|B\|\big(\|f\|+\|g\|\big);\\
&\|\Phi(a*f)-a*\Phi(f)\|= \|B(a*f)-a*B(f)\|\leq \|B\|\|a*f\|+\|a\|_{L(Z)}\|B\|\|f\|\leq C\|a\|\|f\|.
\end{align*}
To check that the extension generated by $\Phi$ is equivalent to the starting one, define $u: Y\oplus_\Phi L_p^0\To Z$ by $u(y,f)=y+L(f)$. Such $u$ is bounded since writing $L=B-\Phi$ one has
$$
\|u(y,f)\|=\|y+Bf-\Phi f\|\leq  \|y-\Phi f\|+\|B f\|\leq \|B\| \big( \|(y,f)\|_\Phi \big).
$$
Therefore, $u$ is a homomorphism and makes commutative the diagram
$$
\xymatrix{
0\ar[r] & Y \ar@{=}[d] \ar[r]^-\imath & Y\oplus_\Phi L_p^0 \ar[d]^{u}\ar[r]^-\pi & L_p^0  \ar[r]\ar[d]^{\text{inclusion}} & 0\\
0\ar[r] & Y \ar[r]^-\jmath & Z \ar[r]^-\varpi & L_p  \ar[r] & 0
}
$$
Extending $u$ to the completion of $ Y\oplus_\Phi L_p^0$ ends the proof.
\end{proof}

Note that if $\Gamma$ is countable (equivalently, if $G$ is metrizable) the proof is simpler since in this case $L_p$ contains a dense free submodule.
The identification of morphisms from the polynomials to modules on which the Fourier transform makes sense is very easy. They are all ``multipliers'': 

\begin{lem}\label{lem:morphisms}
Let $a: \Gamma\To \mathbb C$ be any function. Then the map $\alpha: P(G)\To M(G)$ defined by $\alpha(f)=\mathscr F^{-1}(a\wh f)$ is a morphism of $L_1(G)$-modules and every morphism arises in this way. In particular, every  morphism $ P(G)\To M(G)$ takes values in $P(G)$. 
\end{lem}

\begin{proof}
The first part of the first part is trivial. The converse is easy: if $\alpha: P(G)\To M(G)$ is a morphism of $L_1$-modules and $\mu=\alpha(\gamma)$ then $\gamma*\mu=\mu$, and from here it follows that $\mu = a \gamma dx$ for some $a\in\mathbb C$ depending on $\gamma$.
\end{proof}

\subsection{The Kalton-Peck maps}\label{sec:KPmaps}
We now introduce the quasilinear maps we need to carry out the main construction.
Although our main goal is the construction of centralizers over convolution algebras,
we ask the reader to forget about groups for a moment since they play no role here. So, let $I$ be an ``index'' set (which will be countable in the applications) and consider the spaces $\ell_p(I)$ for $1\leq p\leq\infty$ with their usual norms.

It is clear that each $\ell_p(I)$ is an $\ell_\infty(I)$-module under the pointwise product and that $\ell_p^0(I)$, the subspace of finitely supported elements, is a dense submodule of $\ell_p(I)$ for $p<\infty$.
Please note that the meaning of the superscript $0$ is different in $\ell_p^0(I)$ than in $L_p^0(G)$.

 Given a Lipschitz function $\varphi:\mathbb R\To\mathbb C$ vanishing at zero, the Kalton-Peck map $\Phi:\ell_p^0(I)\To\ell_p(I)$   pointwise defined by
$$
\Phi(f)= f\varphi\left(\log\frac{\|f\|_p}{|f|}\right)
$$ 
is an $\ell_\infty(I)$-centralizer with quasilinear constant $8L_\varphi/e$ and centralizer constant $2L_\varphi/e$, where $L_\varphi$ is the Lipschitz constant of $\varphi$. This appears in \cite{k-p} and \cite{k-comm}; see \cite[Section 3.12]{2c} for a more homogeneous argument. Actually $\Phi$ is {\em strictly symmetric} meaning that it has the following properties:
\begin{itemize}
\item[($\flat$)] $\Phi(f\circ\sigma)= (\Phi f)\circ\sigma$ for every bijection $\sigma:I\To I$.
\item[($\sharp$)] If $u\in\ell_\infty(I)$ is unitary then $\Phi(uf)= u\Phi (f)$.
\end{itemize}
These will play a role later. The fact that $\Phi$ depends not only on $\varphi$ but also on $p$ will not cause any confussion and actually only the case $p=2$ will be used in the main construction.

Now assume that $\sigma: I\To I$ is an injective map, not necessarily surjective. Then $\sigma$ induces a rearrangement operator $R^\sigma$ on every $\ell_p(I)$ sending $e_n$ to $e_{\sigma(n)}$, that is,
$$
R^\sigma(f)(k)=\begin{cases}
f(\sigma^{-1}(k))  & \text{for $k$ in the range of $\sigma$;}\\
0 &\text{otherwise.}
\end{cases}
$$
We close this section with the following remark:

\begin{lem}\label{lem:abc}$\,$
\begin{itemize}
\item[(a)]Every quasilinear map $\Phi:\ell_p^0(I)\To\ell_p(I)$ having property $(\sharp)$ is an $\ell_\infty(I)$-centralizer.

\item[(b)]Every mapping $\Phi:\ell_p^0(I)\To\ell_p(I)$ having property $(\sharp)$ preserves supports in the sense that $\supp(\Phi f)\subset\supp f$ for all $f$.

\item[(c)]
Every mapping $\Phi:\ell_p^0(I)\To\ell_p^0(I)$ having property $(\flat)$ commutes with $R^\sigma$ for any injection $\sigma:I\To I$.
\end{itemize}
\end{lem}

\begin{proof} (a) follows from the fact that every function in the unit ball of $\ell_\infty(I)$ 
can be written as the average of four unitaries.

(b) Pick $f$ and let $u$ be the function that assumes the value 1 on $\supp f$ and $-1$ elsewhere. Then $uf=f$ and $u\Phi f=\Phi f$ implies that  
 $\supp(\Phi f)\subset\supp f$.
 
 (c) is obvious if $I$ is finite. Otherwise for each $f$ we can take a bijection $\varsigma$ that agrees with $\sigma^{-1}$ on $\supp f\cup \supp\Phi f$ and thus $R^\sigma(\Phi f)= (\Phi f)\circ\varsigma= \Phi (f\circ\varsigma)= \Phi (R^\sigma f)$.
\end{proof}

\section{Construction of $L_1$-centralizers}\label{sec:main}
Although some of the ideas of this section could be presented in a more general setting the final argument works only in the compact case since we need the containments $L_p(G)\subset L_q(G)$ for $q\leq p$. So, we invariably assume that $G$ is a compact group with dual group $\Gamma$. We emphasize the fact that $\Gamma$ is discrete by writing $\ell_p(\Gamma)$ instead of the customary $L_p(\Gamma)$.
Note that $\mathscr F$ is an isometry of $L_2^0(G)$ onto $\ell_2^0(\Gamma)$.

We will present three different constructions of $L_1$-centralizers, each with its pros and cons. The first two stem from the same idea: one fixes $p$ and $q$, takes $f\in L_p(G)$, does {\em something} with its Fourier coefficients, and moves back to $L_q(G)$.

\subsection{A global construction}\label{sec:global}
Given a quasilinear map $\Omega: \ell_2^0(\Gamma) \To  \ell_2(\Gamma)$  we 
consider the composition
$$
\xymatrixcolsep{3pc}
\xymatrix{\mho: 
 L_2^0(G) \ar[r]^{\mathscr F} & \ell_2^0(\Gamma) \ar[r]^\Omega & \ell_2(\Gamma) \ar[r]^{\mathscr F^{-1}} & L_2(G) 
}
$$
Obviously $\mho$ is quasilinear, with $Q(\mho)= Q(\Omega)$. Also, for  $1\leq q\leq 2\leq p\leq \infty$ we can consider the composition
$$
\xymatrixcolsep{3pc}
\xymatrix{\mho^{pq}: 
L_p^0(G) \ar[r]^-{\text{inclusion}} & L_2^0(G) \ar[r]^{\mathscr F} & \ell_2^0(\Gamma) \ar[r]^\Omega & \ell_2(\Gamma) \ar[r]^{\mathscr F^{-1}} & L_2(G) \ar[r]^{\text{inclusion}} & L_q(G) 
}
$$
where $\mho^{22}=\mho$. It is clear that every $\mho^{pq}$ is quasilinear, with $Q(\mho^{pq})\leq Q(\Omega)$. As pure maps all $\mho^{pq}$ agree. However, they have different properties depending on the norms one considers in the domain and range. 

Let us notice that if $\Omega$ maps finitely supported sequences to finitely supported sequences  then $\mho$ maps polynomials to polynomials. If, besides, $\Omega$ preserves supports (as do all Kalton-Peck maps and can be assumed for all $\ell_\infty(\Gamma)$-centralizers) then, given characters $\gamma_1,\dots,\gamma_k\in \Gamma$, one has $\mho[\gamma_1,\dots,\gamma_k]\subset [\gamma_1,\dots,\gamma_k]$, where the square brackets denote the linear span.

From a linear point of view $\mho$ is ``the same object'' as $\Omega$; however, things are different from the perspective of the algebras $\ell_\infty(\Gamma), L_\infty(G)$ and the convolution algebras $\ell_1(\Gamma), L_1(G)$ as we now see.
Recall that each $\eta\in \Gamma$ ``translates'' functions on $\Gamma$ by the rule $f_\eta(\gamma)=f(\gamma\eta^{-1})$. We say that $\Omega:\ell_2^0(\Gamma)\To \ell_2(\Gamma)$ commutes with translations if $\Omega(f_\eta)=(\Omega f)_\eta$ for every $\eta$ and every $f$.

\begin{lem}\label{lem:*vs.}
With the same notations as before:
\begin{itemize}
\item[(a)] If $\Omega$ is an $\ell_\infty(\Gamma)$-centralizer and $1\leq q\leq 2\leq p\leq \infty$ then $\mho^{pq}$ is an $L_1(G)$-centralizer, with $C_{L_1(G)}[\mho^{pq}]\leq C_{\ell_\infty(\Gamma)}[\Omega]$.
\item[(b)] $\Omega$ commutes with characters if and only if $\mho^{pq}$ commutes with translations.
\item[(c)] $\Omega$ commutes with translations if and only if $\mho^{pq}$ commutes with characters.
\end{itemize}
\end{lem}
In the following proof we simply write $\mho$ for any of the maps $\mho^{pq}$.
\begin{proof}
a) Actually $\mho$ is even a centralizer over the larger convolution algebra $M(G)$: pick $\mu\in M(G)$ and $f\in L_2^0(G)$. Then
\begin{align*}
\|\mho(\mu*f)-\mu*\mho(f)\|_{L_2(G)}&=
\left\|\F^{-1}\big(\Omega(\wh{\mu\!*\!f})\big)-\mu*\F^{-1}(\Omega\wh{f})\right\|_{L_2(G)}=
\|(\Omega(\wh{\mu}\,\wh f)-\wh \mu\,\Omega(\wh{f})\|_{\ell_2(\Gamma)}\\
&\leq C_{\ell_\infty(\Gamma)}(\Omega)\|\wh\mu\|_{\ell_\infty(\Gamma)}\|\wh f\|_{\ell_2(\Gamma)}\leq C_{\ell_\infty(\Gamma)}(\Omega)\|\mu\|_{M(G)}\|f\|_{L_2(G)}.
\end{align*}

(b) The hypothesis on $\Omega$ means that $\Omega(y c)=y\Omega(c)$ for every $y\in G$ and every finitely supported $c$, where $y:\Gamma\To\mathbb C$ is defined as $y(\gamma)=\gamma(y)$. 
To prove $\implies$
we must check that $\mho(f_y)= (\mho f)_y$ for every $y\in G$ and every polynomial $f$. The point is that $\mathscr F(f_y)=y\,\mathscr F(f)$. Now,
$$
\mho(f_y)=\mathscr F^{-1}\big(\Omega(\wh f_y)\big) = \mathscr F^{-1}\big(\Omega(y\wh f)\big)= 
\mathscr F^{-1}\big(y\Omega(\wh f)\big)= (\Omega f)_y.
$$
To prove the converse note that $\Omega(c)=\mathscr F\mho( \mathscr F^{-1}(c))$ for every finitely supported $c:\Gamma\To\mathbb C$. If $c=\wh f$, we have
$$
\Omega(yc)= \mathscr F\mho( \mathscr F^{-1}(yc))
= \mathscr F\mho( f_y)=\mathscr F\big((\mho f)_y\big)=y \Omega(c).
$$ 
The proof of (c) is formally identical to that of (b).
\end{proof}
Note that the hypotheses on $\Omega$ in (a), (b) and (c) are automatic if $\Omega$ is strictly symmetric since each $y\in G$ is unitary in $\ell_\infty(\Gamma)$.
We now study the triviality of the quasilinear maps $\mho^{pq}$ for $1\leq q\leq 2\leq p \leq \infty$. Since we have the factorization
$$
\xymatrixcolsep{3pc}
\xymatrix{
\mho^{\infty 1}: L_\infty^0(G) \ar[r]^-{\text{inclusion}} & L_p^0(G) \ar[r]^{\mho_{pq}} & L_q(G) \ar[r]^{\text{inclusion}} & L_1(G) 
}
$$
we see that if $\mho^{\infty 1}$ is not trivial then neither is ${\mho^{pq}}$. The following result applies, in particular, when $\Omega:\ell_2^0(\Gamma)\To \ell_2(\Gamma)$ commutes with translations and $\Phi=\mho^{\infty 1}$:

\begin{prop}\label{prop:trivial-bounded} Let $\Phi:  L_\infty^0(G)\To L_1(G)$ be a quasilinear map such that $\Phi(\gamma f)=\gamma\,\Phi(f)$ for every $\gamma\in\Gamma$ and every $f\in  L_\infty^0(G)$. TFAE:
\begin{enumerate}
\item $\Phi: L_\infty^0(G)\To L_1(G)$ is trivial;
\item There is  $\mu\in M(G)$ so that
	$\|\Phi f - f\mu \|_{M(G)} \leq K \|f\|_\infty$ for some $K$ and every $f\in L_\infty^0(G)$.
\item $\Phi$ is bounded from $L_\infty^0(G)$ to $L_1(G)$.
\end{enumerate}
\end{prop}

\begin{proof}  Regarding (2), if $f\in C(G)$ and $\mu\in M(G)$ then $f\mu$ is the measure defined by $(f\mu)(A)=\int_Afd\mu$. Clearly, $\|f\mu\|_{M(G)}\leq \|f\|_\infty\|\mu\|_{M(G)}$.

(1)$\implies$(2)\quad If $\Phi$ is trivial then there is a linear map $L: L_\infty^0(G) \To L_1(G)$ and a constant $K$ so that $\|\Phi(f)-L(f)\| \leq K\|f\|$ for every trigonometric polynomial $f$. 
Let $m$ be an invariant mean for $\Gamma$ and let us (isometrically) move from $L_1(G)$ to $M(G)$. We kindly ask the reader to have a look at Greenleaf's booklet \cite{green} in case they do not know what an invariant mean is. Let us define a {\em new} linear map $\Lambda:  L_\infty^0(G)\To M(G)$ by the formula
$$
\langle \Lambda(f), h\rangle= m_\gamma \left( \int_G \gamma^{-1} L(\gamma f)h dx \right)\qquad(f\in   L_\infty^0(G), h\in C(G)),
$$ 
where the subscript indicates that $m$ is applied to a function of the variable $\gamma$.
We have 
\begin{align*}
\|\Lambda(f)-\Phi(f)\|_{M(G)} & \leq \sup_{\gamma\in\Gamma}\|\gamma^{-1} L(\gamma f)-\Phi(f)\|_{L_1(G)} = 
\sup_{\gamma\in\Gamma}\|\gamma^{-1} L(\gamma f)- \gamma^{-1} \Phi(\gamma f)\|_{L_1(G)}\\
&= \sup_{\gamma\in\Gamma}\| L(\gamma f)- \Phi(\gamma f)\|_{L_1(G)}\leq K\|f\|_{L_\infty(G)}.
\end{align*}

It is very easy to see that $\Lambda(\gamma f)=\gamma \Lambda(f)$ for every $\gamma\in \Gamma$ and every polynomial $f$. Since every polynomial is a linear combination of (finitely many) characters and $\Lambda$ is linear one also has $\Lambda(g f)=g \Lambda(f)$ for all polynomials $f,g$, and so 
$\Lambda(f)=\Lambda(f1_G)= f\Lambda(1_G)$, which yields (2) just taking $\mu=\Lambda(1_G)$.

(2)$\implies$(3)\quad  $\|\Phi(f)\|_{L_1(G)}\leq \|\Phi(f)-f\mu\|_{M(G)}+\| f\mu\|_{M(G)}\leq K\|f\|_\infty+\|\mu\|_{M(G)}\|f\|_\infty$.
\end{proof}
The following result is a quantitative version of the preceding one. To proceed, given a quasilinear map $\Phi: E\To F$ we define
$
\delta(\Phi)=\inf\{\|\Phi-L\|: L \text{ is a linear map from $E$ to $F$}\}.
$  
 Of course $\delta(\Phi)$ is finite if and only if $\Phi$ is trivial, as a quasilinear map. In general   $\delta(\Phi)$ quantifies the ``degree'' of non-triviality of a trivial map $\Phi$.

\begin{cor}\label{cor:distvsbound}
Let $\Phi:  L_\infty^0(G)\To L_1(G)$ be a quasilinear map commuting with the characters of $G$. Then $2\delta(\Phi)\geq \|\Phi\|-\|\Phi(1_G)\|_1$.
\end{cor}

\begin{proof}
We may assume $\delta(\Phi)$ and $\|\Phi\|$ finite. 
Fix $\varepsilon>0$  and let $f_\varepsilon$ be normalized in $L_\infty^0(G)$ so that $\|\Phi(f_\varepsilon)\|_1>\|\Phi\|-\varepsilon$. Also, let $L_\varepsilon: L_\infty^0(G)\To L_1(G)$ be a linear map so that $\|\Phi-L_\varepsilon\|<\delta(\Phi)+\varepsilon$. Use the preceding proposition to obtain a measure $\mu_\varepsilon$ so that $\|\Phi(f)-f\mu_\varepsilon\|_{M(G)}\leq \big(\delta(\Phi)+\varepsilon\big)\|f\|_\infty$ for every polynomial $f$. Now, since $\|\Phi(1_G)-\mu_\varepsilon\|_{M(G)}< \delta(\Phi)+\varepsilon$, we have
$$
\delta(\Phi)+\varepsilon\geq \|\Phi f_\varepsilon- f_\varepsilon\mu_\varepsilon\|\geq 
\|\Phi f_\varepsilon\|_1- \|f_\varepsilon\mu_\varepsilon\|_{M(G)}\geq \|\Phi\|-\varepsilon-\|\mu_\e\|\geq  \|\Phi\|-\varepsilon-\delta(\Phi)-\|\Phi(1_G)\|,
$$
which is enough.
\end{proof}

\subsection*{Nontriviality}\label{sec:circle}
Our next aim is, of course, showing
that this mechanism is able to produce nontrivial quasilinear maps.
Although we suspect that $\mho^{\infty 1}$ (and so $\mho^{pq}$ for $1\le q\le 2\le p\le\infty$) is nontrivial for every nontrivial $\Omega$ we have a proof only for some well-behaved centralizers of Kalton-Peck type. 
Let us say that a Lipschitz function $\varphi: [0,\infty)\To\mathbb R$, with $\varphi(0)=0$, is \emph{essentially concave} if it is non-negative and concave (hence non-decreasing) on some interval $[t,\infty)$. Typical examples are $t\longmapsto t^\alpha$ for $0<\alpha\leq 1$ or $t\longmapsto \log(1+t)$ and also $t\longmapsto at^\alpha-bt^\beta$ for $0<\beta<\alpha\leq 1$ and $a>0$.


\begin{thm}\label{th:main}
Let $G$ be an infinite compact abelian group with dual $\Gamma$.
If $\Omega:\ell_2^0(\Gamma)\To \ell_2^0(\Gamma)$ is the Kalton-Peck map associated to an essentially concave, unbounded function then $\mho^{\infty 1}: C^0(G)\To L_1(G)$ is not trivial and neither is $\mho^{pq}: L_p^0(G)\To L_q(G)$ for  $1\leq q\leq 2\leq p\leq  \infty$.
\end{thm}

In the remainder of the section we short $L_p(G)$ to plain $L_p$. It is clear from Lemma~\ref{lem:*vs.} that $\mho$ is actually an $L_1$-centralizer. Moreover, the observation before Proposition~\ref{prop:trivial-bounded} guarantees it is enough to prove that $\mho^{\infty 1}$ is unbounded. This amounts to finding functions $f$ with controlled $L_\infty$-norm 
such that $\|\mho f\|_1$ is arbitrarily large.

The main feature here is the use of dissociate Sidon sets which include lacunary sets of characters on $\mathbb T$ and the Rademacher functions on Cantor's group $\Delta$.
To be true we came to Sidon sets only after realizing that they provided a common framework for our earlier proofs for some particular cases of the theorem.
A set $\Sigma\subset  \Gamma$ is called {\em dissociate} if $1\notin\Sigma$ and for finitely many $\gamma_1,\dots,\gamma_k$ different elements of $\Sigma$ and integers $n(j)\in\{0,\pm 1,\pm 2\}$ the implication
$$
\prod_{1\leq j\leq k}\gamma_j^{n(j)}=1\quad\implies  \quad \gamma_j^{n(j)}=1\quad\forall j
$$
holds; see \cite[Definition 2.5]{lr} or \cite[37.12 Definition]{hr}. This notion makes sense in any group and actually 
 infinite dissociate sets of exist in every infinite group: \cite[37.26 Remark (b)]{hr} or \cite[Theorem 2.8]{lr}.
Also, $\Sigma\subset \Gamma$ is said to be a \emph{Sidon} set if every $c\in c_0(\Sigma)$ has the form $\wh f|_\Sigma$ for some $f\in L_1$, in which case there exists $S$ (the Sidon constant of $\Sigma$) such that one can choose $f$ so that $\|f\|_1\leq S\|c\|_\infty$.
Every dissociate set is a Sidon set; actually if $\Sigma$ is dissociate, then $\Sigma\cup \Sigma^{-1}$ is a Sidon set (\cite[37.15 Corollary and 37.27]{hr} or \cite[Corollary 2.9]{lr}).
Sidon sets behave in many respects as the Rademacher system $(r_n)_{n\geq 1}$ as witnessed by the following particular case of a result of Pisier \cite[Th\'eor\`eme 2.1]{pi-sidon}: if $\Sigma$ is a Sidon set, then there is a constant $c$ depending only on $S(\Sigma)$ such that 
\begin{equation}\label{eq:pisier}
c^{-1} \left\|\sum\nolimits_{i\leq k}a_i r_i\right\|_{L_p(\Delta)}\leq 
 \left\|\sum\nolimits_{i\leq k}a_i \gamma_i\right\|_{L_p(G)} \leq 
c \left\|\sum\nolimits_{i\leq k}a_i r_i\right\|_{L_p(\Delta)}
\end{equation}
for all $k$, all scalars $a_i$ and all $\gamma_i\in \Sigma$.
In its turn Khintchine's inequalities state that for every finite $p$ there exists constants $A_p, B_p$ so that
\begin{equation}\label{eq:khintchine}
A_p\left(\sum\nolimits_{i\leq k}|a_i|^2\right)^{1/2} \leq 
\left\|\sum\nolimits_{i\leq k}a_i r_i\right\|_{L_p(\Delta)}\leq 
B_p\left(\sum\nolimits_{i\leq k}|a_i|^2\right)^{1/2}.
\end{equation}
The exact values of these constants  were computed by Haagerup in \cite{haagerup}. We shall use later that $A_1=2^{-1/2}$.


We now fix an infinite dissociate set of characters $\Sigma$. Note that, for any two finite subsets $E,F\subset \Sigma$ and functions $\epsilon:E\To\{\pm1\}, \delta:F\To\{\pm1\}$ so that
$$
\prod_{\gamma\in E}\gamma^{\epsilon(\gamma)}= \prod_{\eta\in F}\eta^{\delta(\eta)}.
$$
we have $E=F$ and $\gamma^{\epsilon(\gamma)}= \gamma^{\delta(\gamma)}$ for all $\gamma\in E$. In particular $\epsilon(\gamma)= \delta(\gamma)$ if $\gamma$ does not have order 2 --and the sign of $\epsilon(\gamma)$ is irrelevant if $\gamma$ has order 2. Let us now consider a sequence $(\gamma_n)_{n\geq 1}$ of different elements of $\Sigma$. Given $\alpha>1$, which will be fixed later, and $N\in \mathbb N$ consider the following Riesz product, depending on $N$ and $\alpha$:
\begin{equation}\label{eq:riesz}
f = \prod_{1\leq j \leq N} \left[1+\frac{i}{\alpha\sqrt{N}}\left(\frac{\gamma_j + \gamma_j^{-1}}{2}\right)\right]
\end{equation}
(see \cite[Section~2.4]{lr} for a discussion on Riesz products).
Let $\Gamma_N$ be the set of characters that can be written as
$\gamma=\prod_{j\leq N} \gamma_j^{\e_j}$ with each $\e_j\in\{-1,0,1\}$. For every  $\gamma\in \Gamma_N$ we define the \emph{length} of $\gamma$ (relative to $\Sigma$) as the number of nonzero exponents in that expression, which is unique up to the signs of the exponents of the elements of order 2 by virtue of dissociateness.

To compute the Fourier coefficients of $f$ one should keep track of the characters of order 2.
Passing to a subset if necessary we only need to distinguish two cases, namely:
\begin{itemize}
	\item[$(\dagger)$] $\Sigma$ has no elements of order $2$. 
	\item[$(\ddagger)$] Every element in $\Sigma$ has order $2$.
\end{itemize}
In the first case, the Fourier coefficients of $f$ are, for $\gamma\in \Gamma_N$,
$$ \wh f(\gamma) = \left(\frac{i}{2\alpha\sqrt{N}}\right)^{\ell(\gamma)}$$
and zero otherwise.
Since $\wh f(\gamma)$ only depends on the length of $\gamma$, we have
$$f = f_0 + f_1 + \cdots + f_N, \quad\text{where}\quad f_k = \sum_{\ell(\gamma)=k} \wh f_k(\gamma) \, \gamma = 
\left(\frac{i}{2\alpha\sqrt{N}}\right)^{k} \sum_{\ell(\gamma)=k} \gamma\quad\text{for}\quad0\leq k\leq N.$$
Note that $f_0=1$, that the sum is taken over $\Gamma_N$ and that there are exactly $\binom{N}{k}2^k$ characters of length $k$ in $\Gamma_N$.
Assuming $(\ddagger)$ instead one has
 \begin{equation} \label{eq:rieszII} f= \prod_{1\leq j\leq N} \left(1+\frac{i\, \gamma_j}{\alpha\sqrt{N}}\right)
 \end{equation}
and the Fourier coefficients of $f$ are different: for $\gamma\in \Gamma_N$, we have
$ \wh f(\gamma) = \left(\frac{i}{\alpha\sqrt{N}}\right)^{\ell(\gamma)}
$ 
and zero otherwise. Besides, now there are $\binom{N}{k}$ characters of length $k$ in $\Gamma_N$ and the decomposition we get is slightly different:
$$f = f_0 + f_1 + \cdots + f_N, \quad\text{where}\quad f_k = \sum_{\ell(\gamma)=k} \wh f_k(\gamma) \, \gamma = 
\left(\frac{i}{\alpha\sqrt{N}}\right)^{k} \sum_{\ell(\gamma)=k} \gamma\quad\text{for}\quad 0\leq k\leq N.$$

 The following result collects some useful properties of these functions: 

\begin{lem} \label{lem:riesz}
 For every $N$ and every $0\leq k\leq N$ the following hold:
\begin{enumerate}
	\item[(a$\dagger$)]  If  $\Sigma$ does not contain elements of order $2$, then  $\|f_k\|_{L_2} \leq \frac{1}{\alpha^k} \left(\frac{1}{2^kk!}\right)^{1/2}$, and therefore, $1\leq \|f\|_\infty \leq \|f\|_{L_2} \leq \left(1+\frac{1}{2 \alpha^2 N}\right)^{N/2} \leq e^{1/(4\alpha^2)}$.
	\item[(a$\ddagger$)] If every element of  $\Sigma$  has order $2$, then  $\|f_k\|_{L_2} \leq \frac{1}{\alpha^k} \left(\frac{1}{k!}\right)^{1/2}$, and therefore, $1\leq \|f\|_\infty \leq \|f\|_{L_2} \leq \left(1+\frac{1}{\alpha^2 N}\right)^{N/2} \leq e^{1/(2\alpha^2)}$.
	\item[(b)] In any case, if $k$ is even then $f_{k}$ is real; if $k$ is odd, $f_{k}$ is purely imaginary.
\end{enumerate}
\end{lem}
\begin{proof} (a$\dagger$) In this case the number of characters of length $k$ in $\Gamma_N$ is $\binom{N}{k}2^k$. Therefore:
\[ \|f_k\|_{L_2} = \frac{1}{\alpha^k} \left[\binom{N}{k}(2N)^{-k}\right]^\frac12 = \frac{1}{\alpha^k} \left[\frac{1}{2^k k!} \left(1-\frac{1}{N}\right) \cdots \left(1-\frac{k+1}{N}\right) \right]^\frac12\]
which is, for a fixed $k$, an increasing sequence converging to $\frac{1}{\alpha^k}\left(\frac{1}{2^kk!}\right)^{1/2}$. Also,
\[ \|f\|_{L_2}^2 = \sum_{k=0}^N \binom{N}{k} (2N\alpha^2)^{-k} = \left(1+\frac{1}{2N\alpha^2}\right)^N\leq e^{1/(2a^2)}.\]

(a$\ddagger$) In this case we have
$$
\|f_k\|_2=(\alpha\sqrt{N})^{-k} \binom{N}{k}^{1\over 2}\leq \frac{1}{\alpha^k}\left(\frac{1}{k!}\right)^{1\over 2}\qquad\implies\qquad \|f\|_\infty \leq \|f\|_{L_2} = \left(1+\frac{1}{\alpha^2 N}\right)^{N\over 2} \leq e^{\frac{1}{2\alpha^2}}
$$ 
Part (b) is obvious since $\gamma$ and $\gamma^{-1}$ have the same length.
\end{proof}

\begin{proof}[End of the proof of Theorem~\ref{th:main}]
When $f$ is a Riesz product of the form (\ref{eq:riesz}), its $L_\infty$-norm is bounded by $e^{1/(2\alpha^2)}$, according to the previous lemma. We now show that the $L_1$-norm of $\mho(f)$ goes to infinity with $N$.  Since
	$$\wh f = \wh 1 + \wh{f_1}+\cdots+\wh{f_N}$$
and the $\wh f_k$'s have disjoint support in $\Gamma$ one obtains
\begin{align*} \Omega \wh f &= 
		\wh f \cdot\varphi\left(\log \frac{\|f\|_2}{|\wh f|}\right) =
		\big(\wh 1 + \wh{f_1}+\cdots+\wh{f_N}\big) \cdot\varphi\left(\log \frac{\|f\|_{L_2}}{|\wh 1 + \wh{f_1}+\cdots+\wh{f_N}|}\right) = \\[2mm]
		&= 
		\wh 1 \,\varphi\left(\log \frac{\|f\|_{L_2}}{\wh 1}\right) + \wh{f_1}\,\varphi\left(\log\frac{\|f\|_{L_2}}{|\wh f_1|}\right) + \cdots + \wh{f_N}\,\varphi\left(\log\frac{\|f\|_{L_2}}{|\wh f_N|}\right).
	\end{align*}
	Assuming ($\dagger$) and taking the inverse Fourier transform we see that the imaginary part of $\mho f$ is given by 
	$$
	f_1  \varphi\left(\log\big(\|f\|_{L_2} 2\alpha\sqrt{N}\big)\right)+ f_3  \varphi\left(\log \big({\|f\|_{L_2}}\big(2\alpha\sqrt{N}\big)^3\big)\right)+ \dots = \sum_{ k\leq N  \text{ odd}}  \! f_k  \varphi\left(\log\big(\|f\|_{L_2} (2\alpha\sqrt{N})^k\big)\right)
	$$
The $L_1$-norm of this sum is large because the $L_1$-norm of the first summand is large and the others have small $L_2$-norms. The details are as follows. 
Letting $b=\log\|f\|_2$ and $a=\log(2\alpha\sqrt{N})$, we have  	
\begin{align*}
\|\mho&(f)\|_{L_1}\geq 
 \left\|\sum_{ k\leq N  \text{ odd}} \varphi(b+ka)\, f_k \right\|_{L_1}
 \geq \varphi(b+a)\|f_1\|_{L_1}- \left\|\sum_{ k\geq 3 \text{ odd}} \varphi(b+ka)\, f_k \right\|_{L_1}\\
& \geq   \varphi(b+a)\|f_1\|_{L_1}-\sum_{ k\geq 3 \text{ odd}}  \varphi(b+ka) \|f_k\|_{2}
 \geq   \varphi(b+a)\|f_1\|_{L_1}-\sum_{ k\geq 3 \text{ odd}}  \big(\varphi(b)+k\varphi(a)\big) \|f_k\|_{2}.
\end{align*}
for each given $\alpha$ provided $N$ is large enough so as to $\varphi$ is concave on $[b+a,\infty)$.
Now, as $\Sigma\cup \Sigma^{-1}$ is a Sidon set, applying first the Pisier inequality and then Khintchine inequality we get a constant $C$, depending only on $S(\Sigma\cup \Sigma^{-1})$ so that, for all $\alpha$ and all $N$,
$$
\|f_1\|_{L_1} \geq \frac{1}{\sqrt{2}C\alpha} \quad\text{since}\quad f_1=\frac{i}{2\alpha\sqrt{N}}\sum_{1\leq j\leq N}(\gamma_j+\gamma_j^{-1}).
$$
%
%
To get an upper bound of
the ``remainder'', note that $b\leq L(\varphi)/(4\alpha^{2})$, so
\begin{align*}
\sum_{ k\geq 3 \text{ odd}}  \big(\varphi(b)+k\varphi(a)\big) \|f_k\|_{2}
\leq \left(\frac{L_\varphi}{4\alpha^2}+\varphi(a)\right)\sum_{k\geq3\text{ odd}} k \|f_k\|_{L_2}
\end{align*}
It is therefore enough to see that the sum is of the order $\alpha^{-3}$. Using that $k! > 2^k $ for $k>3$, we obtain
$$ \sum_{ k\geq 3  \text{ odd}}  k \|f_k\|_{L_2} 
\leq 
\sum_{ k\geq 3  \text{ odd}} \frac{k}{\alpha^k}\sqrt{\frac{1}{2^kk!}} 
\leq 
\frac{\sqrt{3}}{4\alpha^3} + \sum_{ k\geq 5  \text{ odd}} \frac{k}{(2\alpha)^{k}} 
\leq
\frac{\sqrt{3}}{4\alpha^3} + \frac{17}{216\alpha^3}\leq \frac{1}{\alpha^3}.
$$
Combining,
$$
\|\mho(f)\|_{L_1}\geq \frac{\varphi(a)}{\sqrt{2}C\alpha} -\frac{ \varphi(a)}{\alpha^3}-\frac{L_\varphi}{4\alpha^5}
=
\frac{\varphi(a)}{\sqrt{2} \alpha}\overbrace{\left[ \frac{1}{C} -\frac{\sqrt{2}}{\alpha^2}\right]}^{(\star)}-\frac{L_\varphi}{4\alpha^5}.
$$
Thus, we only need to take $\alpha$ so that $(\star)>0$ and then let $N\To\infty$ (hence $a, \varphi(a)\To\infty$). This ends the proof under the assumption ($\dagger$).

If, instead, one assumes ($\ddagger$) the imaginary part of $\mho f$ is
$$ \sum_{ k\leq N  \text{ odd}}  \! f_k  \varphi\left(\log\big(\|f\|_{L_2} (\alpha\sqrt{N})^k\big)\right).
$$
One then has
\begin{align*}
\|\mho&(f)\|_{L_1}
 \geq   \varphi(c+d)\|f_1\|_{L_1}-\sum_{ k\geq 3 \text{ odd}}  \big(\varphi(c)+k\varphi(d)\big) \|f_k\|_{2}.
\end{align*}
where $c=\log\|f\|_2\leq L_\varphi/(2\alpha^2)$ and $d=\log\big(\alpha\sqrt{N}\big)$ for each $\alpha\geq 1$ if $N$ is large enough to guarantee  concavity of $\varphi$ on $[c+d,\infty)$. Using the estimates in (a$\ddagger$) an entirely analogous reasoning yields unboundedness of $\mho$:
first, we have 
$$
f_1=\frac{1}{\alpha\sqrt{N}}\sum_{j\leq N}\gamma_j\qquad\implies \qquad\|f_1\|\geq \frac{1}{\alpha C},
$$
where $C$ depends only on the Sidon constant of $\Sigma=\Sigma^{-1}$. The other inequality we need is
$$
\sum_{k\geq 3 \text{ odd}}k\|f_k\|_2= \frac{3}{\sqrt{6}\alpha^3}+ 
\sum_{k\geq 5 \text{ odd}} \frac{k}{\alpha^k \sqrt{k!}}\leq 
\frac{3}{\sqrt{6}\alpha^3}+ 
\sum_{k\geq 5 \text{ odd}} {\alpha^{-k}}= 
\frac{3}{\sqrt{6}\alpha^3}+ \frac{1}{\alpha^3(\alpha^2-1)}.
$$
Combining,
\begin{equation} \label{eq:esti}
	\|\mho(f)\|_{L_1}\geq \frac{\varphi(d)}{\alpha}\bigg[\overbrace{\frac{1}{C}- \frac{1}{\alpha^2}\left(\frac{3}{\sqrt{6}}+ \frac{1}{\alpha^2-1}\right)}^{(\star)}\bigg]-
	\frac{L_\varphi}{2\alpha^5} \left(\frac{3}{\sqrt{6}}+ \frac{1}{\alpha^2-1}\right)
\end{equation}
and taking $\alpha$ so that $(\star)>0$ ends the proof.
\end{proof}

\subsection{Other ``quasimultipliers''} 
We now consider the existence of extensions of $L_1$-modules $0\To L_q\To Z\To L_p\To 0$ for arbitrary $1\le p,q\le\infty$. Our conclusion will be that they exist as long as the general theory allows it.

Let $\Sigma\subset\Gamma$ be a Sidon set. Combining (\ref{eq:pisier}) with Khintchine inequality
it is clear that the ``summation'' operator $S: \ell_2(\Sigma)\To L_p(G)$ given by $S(c)=\sum_\gamma c(\gamma)\gamma$ is an isomorphic embedding for finite $p$. An obvious duality argument then shows that the operator  $M:L_q(G)\To \ell_2(\Sigma)$ given by $M(f)=\wh f|_\Sigma$ is bounded and surjective for $q>1$ and also that $MS={\bf I}_{\ell_2(\Sigma)}$ so that the range of $S$ is a complemented subspace of $L_p(G)$ for $1<p<\infty$.

\begin{prop}\label{prop:LpLq} Let $\Sigma\subset \Gamma$ be a Sidon set and let $\Omega:\ell_2^0(\Sigma)\To \ell_2(\Sigma)$ be a nontrivial $\ell_\infty(\Sigma)$-centralizer. Given $p\in(1,\infty], q\in[1,\infty)$ the map
$\mho_\Sigma^{pq}: L_p^0(G)\To L_q(G)$ defined by  $\mho_\Sigma^{pq}(f)=S\Omega(\wh f|_\Sigma)$ is a nontrivial $L_1(G)$-centralizer.
\end{prop}

\begin{proof}[Sketch of the Proof]
The proof that $\mho_\Sigma^{pq}$ is an $L_1$-centralizer is roughly the same as that of Lemma~\ref{lem:*vs.}(a), using the identities $M(a*f)=(Ma)(Mf)$ for $a,f\in L_1$. In particular, note that if $f$ is a polynomial we have $a*S(c)=S((Ma)c)$ for $a\in L_1$ and finitely supported $c$. The nontrivial character of  $\mho_\Sigma^{pq}$ is clear when $1<p,q<\infty$: if $\mho_\Sigma^{pq}$ is trivial as a quasilinear map then so is the composition 
$$
\xymatrixcolsep{4pc}
\xymatrix{
\Omega: \ell_2^0(\Sigma) \ar[r]^S_{\text{bounded}} & L_p^0(G) \ar[r]^{\mho_\Sigma^{pq}} & L_q(G) \ar[r]^M_{\text{bounded}} & \ell_2(\Sigma),
}
$$
 which is not the case. 

For finite $p$ and $q=1$  one can use the following argument: since $S:\ell_2(\Sigma)\To L_p$ is bounded it suffices to see that $S\circ\Omega=\mho_\Sigma^{pq}\circ S$ is not trivial from $\ell_2(\Sigma)$ to $L_1$. But $S:\ell_2\To L_1$ is an (isomorphic) embedding and so $Z(\Omega)$ is isomorphic to a subspace of
$Z(S\circ\Omega)$: just map each $(c,d)$ in $\ell_2(\Sigma)\oplus_\Omega \ell_2^0(\Sigma)$ to $(S(c), d)$ in  $L_1(G)\oplus_{S\circ\Omega}\ell_2^0(\Sigma)$. This implies that $Z(S\circ\Omega)$ cannot be isomorphic to the direct sum $L_1(G)\oplus\ell_2(\Sigma)$ since this space has cotype 2 and $Z(\Omega)$ does not. 

The case $p=\infty, 1<q<\infty$ follows by duality. The case $p=\infty, q=1$ (which  is closely related to the material of Section~\ref{sec:main} and implies all cases) follows from the proof of 
\cite[Theorem 5.1]{ccky} and requires some skill in absolutely summing operators. The argument is given later in the proof of Proposition~\ref{prop:ccky} for the sake of completeness.
%
\end{proof}
Though nontrivial, these ``new'' constructions have much less symmetries than that appearing in Section~\ref{sec:global}: $\mho_\Sigma$ cannot commute with the product by characters unless $G$ is finite, although it  commutes with the translations of $G$ if $\Omega$ commutes with the unitaries of $\ell_\infty(\Sigma)$. 

Finally, the case $p=1$ is resolved with quite general arguments:

\begin{prop}\label{prop:ExtL1}
$\Ext_{L_1}(L_1, Y)=0$ in the following cases: $Y$ is a dual module, in particular when $Y=L_p$ for $1<p\leq\infty$, or if $Y=L_1, C(G)$.
\end{prop}

\begin{proof} The first part is clear: if $Y$ a dual Banach space, then each exact sequence  of Banach spaces $0\To Y\To Z\To L_1\To 0$ splits (in the category  of Banach spaces) by Lindenstrauss' lifting \cite{lind} and then the amenability of $L_1$ implies the same in the category of $L_1$-modules.

The case $Y=L_1$ is as follows: we know from Lemma~\ref{lem:all} that each self-extension of $L_1$ comes from a centralizer $\Phi:L_1^0\To L_1$. If we consider $L_1$ as a submodule of $M(G)$, then the composition
$$
\xymatrixcolsep{3.75pc}
\xymatrix{
L_1^0\ar[r]^\Phi & L_1 \ar[r]^-{\text{``inclusion''}} & M(G)
}
$$
is a centralizer, which is necessarily trivial since $\Ext_{L_1}(L_1,M(G))=0$ by the previous case. Hence there is a morphism $\phi:L_1^0\To M(G)$ such that $\|\Phi(f)-\phi(f)\|_{M(G)}\leq M\|f\|_{L_1}$ for every $f\in L_1^0$. By Lemma~\ref{lem:morphisms}, $\phi$ takes values in $L_1$ and so $\Phi-\phi$ is bounded from $L_1^0$ to $L_1$. This shows that $\Ext_{L_1}(L_1,L_1)=0$.

The case $Y=C(G)$ is analogous, replacing $M(G)$ by $L_\infty(G)$.
\end{proof} 
It is not true in general that $\Ext_{L_1}(X, M(G))=0$ implies  $\Ext_{L_1}(X, L_1)=0$.
To see this consider the inclusion $L_1\To M(G)$ which has complemented range thanks to Lebesgue decomposition theorem (see \cite[Theorem 4.3.2]{cohn}), i.e., the sequence
\begin{equation}\label{eq:L1M}
\xymatrix{
0\ar[r] &L_1\ar[r] & M(G) \ar[r] &M(G)/L_1 \ar[r] & 0
}
\end{equation}
splits in the category of Banach spaces. Set $X= M(G)/L_1$. Then $X$ is an $\mathscr L_1$-space and so $\Ext_{L_1}(X, M(G))=0$ since $\Ext(X, M(G))=0$ (again by Lindenstrauss's lifting) and $L_1$ is amenable. However $\Ext_{L_1}(X, L_1)\neq 0$ because (\ref{eq:L1M}) does not split in the category of $L_1$-modules. Indeed, if $P:M(G)\To L_1$ were an $L_1$-module projection and $f=P(\delta_e)$, where $\delta_e$ is the unit mass at the unit of $G$, then $
g=Pg=P(g*\delta_e)=g*f
$ for every $g\in L_1$,
which cannot be unless $G$ is discrete (hence finite). 

The algebra $L_1$ {\em itself} is not a projective module, as can be seen in the following {\em generic} examples. Let $\fin(\Gamma)$ be family of finite subsets of $\Gamma$ and for each $F\in \fin(\Gamma)$ let $P_F\subset L_1$ be the subspace of polynomials spanned by $F$. Consider the amalgam $\ell^1(\fin(\Gamma), P_F)$ with the obvious structure of $L_1$-module: $a*(f_F)_F= (a*f_F)_F$. The sum operator $\operatorname{sum}(f_F)=\sum_F f_F$ is a surjective homomorphism from $\ell^1(\fin(\Gamma), P_F)$ to $L_1$ and we have an extension
$$
\xymatrixcolsep{3pc}
\xymatrix{
0\ar[r] &\ker(\operatorname{sum})\ar[r] & \ell^1(\fin(\Gamma), P_F) \ar[r]^-{\text{sum}} &L_1 \ar[r] & 0
}
$$
This sequence does not split because the space $L_1$, which lacks the Radon-Nikod\'ym property, cannot be isomorphic to a subspace of $\ell^1(\fin(\Gamma), P_F)$, which has it.

A related example is the following. With the same notations, consider the amalgam of bounded families $c(\fin(\Gamma), P_F)=\{(f_F): \lim_F f_F \text{ exits in } L_1\}$ with the sup norm and the obvious $L_1$-module structure. Since the polynomials are dense in $L_1$ the limit operator $\operatorname{lim}: c(\fin(\Gamma), P_F)\To L_1$, which is clearly a homomorphism, is onto and we have an extension
$$
\xymatrixcolsep{3pc}
\xymatrix{
0\ar[r] & c_0(\fin(\Gamma), P_F) \ar[r] & c(\fin(\Gamma), P_F) \ar[r]^-{\text{lim}} &L_1 \ar[r] & 0.
}
$$
This sequence splits in the Banach category because $L_1$ has the BAP. However if $S: L_1\To c(\fin(\Gamma), P_F)$ were a right inverse homomorphism for the limit map, then writing $S(f)=(S_F(f))$ one easily infers that $S_F(f)=\sum_{\gamma\in F}\wh f(\gamma)\gamma$ and that $\|S_F\|\leq \|S\|$ for all finite $F\subset \Gamma$, which implies that $G$ is finite.


\subsection{Quasilinear maps commuting with translations and $L_1$-centralizers} 
It this section we prove that $L_\infty$-centralizers on $L_p$ that commute with translations are also $L_1$-centralizers when $1<p<\infty$.
Let, once again $G$ be a compact group and write $L_p^\infty=L_p\cap L_\infty$ for the subspace of all essentially bounded functions in $L_p$, which is a dense $L_\infty$-submodule of $L_p$ for all $p$ as well as an $L_1$-submodule for $1\leq p<\infty$.
As the reader can imagine we say that
a (quasilinear) map $\Phi: L_p^\infty\To L_p$ commutes with translations if $\Phi(f_y)=(\Phi f)_y$ for every $y\in G$ and every $f\in L_p^\infty$. 
The simplest examples are provided by
the ``continuous'' versions of the Kalton-Peck maps: if $\varphi:\R\To\mathbb C$ is a Lipschitz funtion vanishing at $0$ we define $\Phi(f)=f\varphi\big(\log\big(\|f\|_p/|f| \big)	\big)$ pointwise. These maps are actually $L_\infty$-centralizers and it is not hard to see that $\Phi$ is trivial if and only if $\varphi$ is bounded on $(-\infty, 0]$. See \cite[Theorem 3.1]{k-comm} for more general examples and note that, in view of the results in \cite{c-ma}, there is no point in considering two different exponents $p,q$ in the ensuing result.

\begin{prop}\label{prop:LooL1}
Let $1<p<\infty$.
Every $L_\infty$-centralizer  $\Phi: L_p^\infty\To L_p$ that commutes with translations is also an $L_1$-centralizer and, therefore, $Z(\Phi)$ is both an $L_\infty$-module and an $L_1$-module.
\end{prop}

\begin{proof}[Sketch of the Proof]
The proof consists in showing that $L_p\oplus_\Phi L_p^\infty$ is an $L_1$ module under the product $a*(g,f)=(a*g,a*f)$.  This implies that $\Phi$ is an $L_1$-centralizer; see Lemma~\ref{lem:modiffcen}.
 We proceed as in the standard proofs of the inequality   $\|a*f\|_B\leq \|a\|_1\|f\|_B$ when $B$ is a homogenenous Banach space, using vector-valued integration; see for instance Katznelson \cite[p. 203]{katz}.
 
Let us first check that:
\begin{itemize}
\item[(a)] The inclusion $L_p\oplus_\Phi L_p^\infty\To L_1\oplus_1 L_1$ is continuous.
\item[(b)] The quasinorm $\|(\cdot,\cdot)\|_\Phi$ is equivalent to a norm on $L_p\oplus_\Phi L_p^\infty$.
\item[(c)] $(g_y,f_y)\To (g_x,f_x)$ in $L_p\oplus_\Phi L_p^\infty$ as $y\To x$ in $G$.
\end{itemize}
To see (a) pick $1<q<p$ and consider $\Phi$ as an $L_\infty$-centralizer $L_p^\infty\To L_q$. By the main result of \cite{c-ma}, $\Phi$ is then trivial and since every morphism of $L_\infty$ modules $L_p^\infty\To L_q$ has the form $f\longmapsto hf$ for some $h\in L_q$ we have $\|\Phi(f)-hf\|_q\leq M\|f\|_p$ for some constant $M$. Averaging over $G$, using the symmetries of $\Phi$ and the reflexivity of $L_q$ one obtains that $\|\Phi(f)-cf\|_q\leq M\|f\|_p$, where $c=\int_Gh(x)dx$, and we are done. (b) is clear since $L_p$ is a $\mathscr K\!$-space for $p>1$. To prove (c) we first observe that since $(g,f)\longmapsto (g_y,f_y)$ is an ``isometric'' automorphism of $L_p\oplus_\Phi L_p^\infty$ it suffices to consider the case in which $x$ is the unit of $G$ and then prove that $(g_y,f_y)\To (g,f)$ for $(g,f)$ in any dense subset of $L_p\oplus_\Phi L_p^\infty$. Next, we observe that the points $(g,f)$ with $f$ continuous form a dense set: pick $(g,f)\in L_p\oplus_\Phi L_p^\infty$ and $\varepsilon>0$. Take $f'\in C(G)$ such that $\|f-f'\|<\varepsilon$ and then set $g'=g-\Phi(f-f')$. We have 
$
\|(g,f)-(g',f')\|_\Phi=\|g-g'-\Phi(f-f')\|_p+\|f-f'\|_p<\varepsilon.
$
To finish, observe that $\Phi$ is bounded from $L_\infty$ to $L_p$ (this is almost trivial) and so $\|\Phi(f)\|_p\leq C\|f\|_\infty$ for some constant $C$ independent of $f$. Take $(g,f)\in L_p\oplus_\Phi L_p^\infty$ with $f$ continuous. One has
$$
\|(g_y,f_y)- (g,f)\|_\Phi=\|g_y-g-\Phi(f_y-f)\|_p+\|f_y-f\|_p.
$$
Now, if $y$ tends to the identity of $G$, then $\|g_y-g\|_p$ and $\|f_y-f\|_p$ go to zero and $\|\Phi(f_y-f)\|_p\leq C\|f_y-f\|_\infty\To 0$ since $f$ is uniformly continuous.

To complete the proof, pick $f\in L_p^\infty, g\in L_p, a\in L_1$.
 We want to see that
\begin{equation}\label{eq:bochner}
(a*g,a*f)=
 {\int_Ga(y)(g_y, f_y)\,dy}
\end{equation}
where the integral is taken in $Z(\Phi)$ 
in the Bochner sense. The function $y\in G\longmapsto a(y)(g_y,f_y)\in L_p\oplus_\Phi L_p^\infty$ is measurable since $a$ is and $y\longmapsto (g_y,f_y)$ is continuous. To check integrability,
let $\|(\cdot,\cdot)\|$ be a norm on $ L_p\oplus_\Phi L_p^\infty$ so that $\|(\cdot,\cdot)\|\leq \|(\cdot,\cdot)\|_\Phi\leq  M\|(\cdot,\cdot)\|$. Then 
$$
\int_G\|a(y)(g_y, f_y)\|\,dy =  \int_G|a(y)|\|(g_y, f_y)\|\,dy\leq 
\int_G|a(y)|\|(g_y, f_y)\|_\Phi\,dy= \|a\|_{L_1}\|(g,f)\|_\Phi.
$$
Thus the integral in (\ref{eq:bochner}) makes sense in $Z(\Phi)$ and 
$
\big\|\int_G a(y)(g_y, f_y)\,dy\big\|_\Phi\leq M \|a\|_{L_1}\|(g,f)\|_\Phi
$. Finally, as the formal inclusion of $L_p\oplus_\Phi L_p^\infty$ into $L_1\oplus L_1$ is continuous one has
$$
\overbrace{\int_Ga(y)(g_y, f_y)\,dy}^{\text{integration in $Z(\Phi)$}}=
\overbrace{\int_Ga(y)(g_y, f_y)\,dy}^{\text{integration in $L_1\oplus L_1$}}= \Big(\overbrace{\int_Ga(y)g_y\,dy}^{\text{integration in $L_1$}}, \overbrace{\int_Ga(y)f_y\,dy}^{\text{integration in $L_1$}}\Big)=(a*g,a*f)\qedhere
$$
\end{proof}

We do not known if every quasilinear map $L_p^\infty\To L_q$ that commutes with translations is an $L_1$-centralizer for $1<p,q<\infty$.



\section{$L_1$-centralizers on the Hardy classes}

When specialized to $G=\mathbb T$ the preceding constructions provide extensions of the Hardy classes in the category of $L_1$-modules. We refer the reader to Duren \cite[Chapter~2]{duren} for the definition and properties of the Hardy classes $H_p$ and much more. For our purposes we can define
$$
H_p=\{f\in L_p: \wh f(n)=0 \text{ for all } n<0\}\qquad(1\leq p\le\infty)
$$
and we set $H_p^0=H_p\cap L_p^0$.
If $1\leq p<\infty$, then $H_p$ agrees with the closure  of the ordinary polynomials $a_0+a_1z+\cdots+a_kz^k$ (with $k\geq 0$) in $L_p=L_p(\mathbb T)$.
For $p=\infty$ the closure of the polynomials in $L_\infty(\mathbb T)$ is the disc algebra, denoted by $A$, which is much smaller than $H_\infty$. 

In the following result we say that an $L_\infty$-centralizer $\Phi:L_p^\infty\To L_p$ (or $L_p^0\To L_p$)  is {\em symmetric} if for every measure preserving
bijection $\sigma:G\To G$ one has $\Phi(f\circ\sigma)=   (\Phi f)\circ\sigma$. Also, we denote by $\gamma_n:\T\To\T$ the character defined by $\gamma_n(z)=z^n$.

\begin{thm}$\,$
\begin{itemize}
\item[{\rm (a)}] Let $\Sigma\subset \mathbb Z_+$ be a Sidon set and let $\Omega:\ell_2^0(\Sigma)\To \ell_2(\Sigma)$ be a nontrivial $\ell_\infty(\Sigma)$-centralizer. Given $p\in(1,\infty), q\in[1,\infty)$, the map
$\mho_\Sigma^{pq}: H_p^0\To H_q$ defined by  $\mho_\Sigma^{pq}(f)=S_q\Omega(\wh f|_\Sigma)$ is a nontrivial $L_1$-centralizer.

\item[{\rm (b)}] Let $\Omega: L_p^0\To L_p$ be a  symmetric $L_\infty$-centralizer, $1<p<\infty$. Then there is an $L_1$-centralizer $\Omega_+:  H_p^0\To H_p$ such that $\|\Omega(f)-\Omega_+(f)\|_p\leq M\|f\|_p$ for some constant $M$ and all $f\in H^0_p$. Any other quasilinear map satisfying the corresponding estimate is strongly equivalent to $\Omega_+$. Moreover, if $1<p<\infty$, then $\Omega_+$ is trivial (as a quasilinear map) if and only if $\Omega$ is. 

\item[{\rm (c)}] Let $\varphi$ be an essentally concave unbounded Lipschitz function and let $\Omega:\ell_2^0(\mathbb N)\To \ell_2(\mathbb N)$ be the induced Kalton-Peck map. Then, for every $1\leq q\leq 2\leq p\leq \infty$ the map $\mho^{pq}: H_p^0\To H_q$ defined by  $\mho^{pq}(f)=S_q\Omega(\wh f)$ is a nontrivial $L_1$-centralizer.
\end{itemize}
\end{thm}

\begin{proof}

Part (a) is exactly as the corresponding cases of Proposition~\ref{prop:LpLq}; note that we are now assuming $1<p<\infty$.

(b) The key point is a result of Kalton stating that the restriction of such an $\Omega$ to $H_p^0$ is ``almost'' $H_p$-valued in the sense that there is a constant $M$ so that for every $f\in H^0_p$ there is $g\in H_p$ such that $\|\Omega(f)-g\|_p\leq M\|f\|_p$; see \cite[Theorem 7.3]{k-comm}. 
Since the Riesz projection $R(f)=\sum_{n\geq 0}\wh f(n)\gamma_n$ is a bounded $L_1$-homomorphism on $L_p$ when $p\in(1,\infty)$, it follows that $\Omega_+(f)=R\Omega(f)$ is an $L_1$-centralizer from $H^0_p$ to $H_p$. Next, we show that if $\Omega_+$ is trivial, then so is $\Omega$. Clearly, $\Omega_+$ is trivial if and only if  $\Omega: H^0_p\To L_p$ is. 
 Put
$
\tilde{H}_p=\ker R=\{f\in L_p: \wh f(k)=0 \text{ for all }k\geq 0\}
$
and observe that the map $\sigma\in L(L_p)$ defined by $(\sigma f)(z)=\overline{z} f(\overline{z})$ is an isometry of $H_p$ onto $
\tilde{H}_p$ and that the hypotheses on $\Omega$ imply that $\sigma\circ\Omega=\Omega\circ\sigma$ and thus $\Omega|_{H^0_p}$ is trivial if and only if $\Omega|_{\tilde{H}^0_p}$
is trivial. Since $\Omega\approx \Omega\circ R+  \Omega\circ({\bf I}- R)$ the proof is done.

Part (c) is a direct consequence of Theorem~\ref{th:main} and the following.
\end{proof}


\begin{prop}\label{prop:trivial-bounded-T}$\;$
\begin{itemize}
\item[{\rm (a)}]
Let $\Phi:  C^0(\mathbb T)\To L_1$ be a quasilinear map such that $\Phi(\gamma_n f)=\gamma_n\Phi(f)$ for every $n\in\mathbb Z$ and every $f\in  C^0(\mathbb T)$. Then $\Phi$ is bounded if and only if its restriction to $A^0$ is bounded.

\item[{\rm (b)}] Let $\Phi:  A^0\To L_1$ be a quasilinear map such that $\Phi(\gamma_n f)=\gamma_n\Phi(f)$ for every $n\geq 0$ and every $f\in  A^0$. Then $\Phi$ is trivial if and only if it is bounded.
\end{itemize}
\end{prop}

\begin{proof}
(a) Actually $\|\Phi: C^0(\mathbb T)\To L_1\|= \|\Phi: A^0\To L_1\|$.
Indeed, for each $f\in C^0(\mathbb T)$ there is $k\in \mathbb N$ such that $\gamma_kf\in A^0$. Now, $\|\Phi(f)\|_1=\|\gamma_k\Phi(f)\|_1=\|\Phi(\gamma_kf)\|_1= $ and $\|f\|_\infty=\|\gamma_kf\|_\infty$.

(b) Assume $\Phi$ trivial so that $\|\Phi(f)-L(f)\|_1\leq K\|f\|_\infty$ for some linear map $L$, a constant $K$, and all $f\in A^0$. We consider $L_1$ inside $M(\T)=C(\T)^*$. Let $m\in \ell_\infty(\Z_+)^*$ be an invariant mean for the {\em semigroup} $\Z_+$ (AKA a Banach limit). Then, since $\|\Phi(f)- \gamma_n^{-1}L(\gamma_nf)\|_1\leq K\|f\|_\infty$
if we define $\Lambda:A^0\To M(\T)$ by letting
$
\langle \Lambda(f), g\rangle= m_n \big(\langle \gamma^{-1}_n L(\gamma_n f), g  \rangle \big)$ for $ g\in C(\T)$. It is easily seen that $\|\Phi-\Lambda\|\leq K$ and that $\gamma_n\Lambda(f)=\Lambda(\gamma_n f)$ for every $n\geq 0$ and all $f\in A^0$, from where it follows that $\Lambda$ has the form $\Lambda(f)=f\mu$, where $\mu=\Lambda(1_\T)$ and that $\Lambda$, and thus $\Phi$, is bounded.
\end{proof}

We remark that the only obstruction to establishing the case $p=\infty$ in the first part of the preceding theorem is that 
we cannot guarantee the surjectivity of the map $f\in A\longmapsto \wh f|_\Sigma\in \ell_2(\Sigma)$ since one has the following ``analytic'' version of \cite[Proof of Theorem~5.1]{ccky}:

\begin{prop}\label{prop:ccky}
Let $0\To H \To Z\To H'\To 0$ be a nontrivial exact sequence, where $H,H'$ are Hilbert spaces. If $J:H\To L_1$ is an embedding and $Q: A\To H'$ is a quotient map, then the extension obtained by taking first pushout with $J$ and then pullback with $Q$ does not split. 
\end{prop}

\begin{proof}[Sketch of the proof] The relevant diagram is
$$\xymatrix{
0\ar[r] & H \ar[d]_J \ar[r] & Z  \ar[d] \ar[r] & H'\ar[r]\ar@{=}[d] &0\\
0\ar[r] & L_1 \ar[r]^\imath \ar@{=}[d] & \PO \ar[r]^-\pi & H'\ar[r] &0\\
0\ar[r] & L_1 \ar[r] & \PB \ar[u] \ar[r]^-\varpi & A \ar[u]_Q\ar[r] &0
}
$$
The relevant information is that
\begin{itemize}
\item $Z$ does not have cotype $2$ (otherwise, the first row splits by the argument in the last three lines of the proof).
\item $\PO$ has finite cotype (since both $L_1$ and $H'$ do).
\end{itemize}
The first point already implies that the second row does not split: the space $\PO$, which contains an isomorphic copy of $Z$, cannot be isomorphic to $L_1\oplus H'$, which has cotype 2.

Assume that the lower sequence splits. Then $Q$ has a lifting $L:A\To\PO$. Since $\PO$ has finite cotype $L$ is $p$-summing for some  $2<p<\infty$ (\cite[11.14 Theorem]{DJT} plus \cite[Corollary~13]{woj}) and by the Pietsch factorization theorem \cite[p. 45]{DJT}
$L$  factors through some $L_p(\mu)$ in the form
$$
\xymatrix{
L: A\ar[r]^\alpha & L_p(\mu)\ar[r]^\beta & \PO
}
$$
Now, $Q=\pi\beta\alpha$ and the map $\pi\beta$ is onto, so we have a commutative diagram
$$\xymatrix{
0\ar[r] & L_1 \ar[r]^\imath & \PO \ar[r]^-\pi & H'\ar[r] &0\\
0\ar[r] & \ker({\varpi\beta})\ar[r]\ar[u]^u & L_p(\mu) \ar[u]^\beta \ar[r] & H' \ar@{=}[u]\ar[r] &0
}
$$
Finally since $L_1$ has cotype 2 and $L_p(\mu)$ has type 2, Maurey's extension theorem \cite[p.~246]{DJT} implies that $u$ extends to an operator $U:L_p(\mu)\To L_1$, and the upper row of the preceding diagram splits, which is not the case.
\end{proof}

\section{A finer analysis based on Cantor's group}\label{sec:cantor}
In this section we explore the peculiarities of the construction of Section~\ref{sec:global} for Cantor's group. We will show that the $L_1$-centralizers $\mho$ have 
a full range of ``spatial'' symmetries.

Let $\Delta=\{-1,1\}^{\mathbb N}$ be Cantor's group and $D$ its dual group.
The $n$-th Rademacher function $r_n:\Delta\To \mathbb T$ is just evaluation at $n$: $r_n(t)=t(n)$. It is clear that the Rademachers are characters and actually every character is a finite product of Rademachers, i.e., a Walsh function. Let us denote by $\fin$ the set of all finite subsets of $\N$, including the empty set.
Each $a\in\fin$ has a character associated with it, namely the Walsh function
\begin{equation}\label{eq:wa}
w_a=\prod_{k\in a}r_k, \qquad w_a(x)=\prod_{k\in a}x(k)
\end{equation}
with $w_\varnothing=1_\Delta$. One then has $D=\{w_a: a\in\fin\}$, with the discrete topology, and $w_a w_b= w_c$, where $c=a\vartriangle b$ is the symmetric difference between $a$ and $b$. We shall identify $D$ and $\fin$ when it suits us without further mention.

From now on, $\Omega$ denotes a strictly symmetric quasilinear map on $\ell_2^0(D)$ and $\mho^{pq}:L_p^0(\Delta)\To L_q(\Delta)$ is as in Section~\ref{sec:global}, that is, $\mho^{pq}(f)=\mathscr F^{-1}\Omega(\wh f)$ where $1\leq q\leq 2\leq p\leq\infty$, with $\mho^{22}$ shortened to plain $\mho$. Recall that, as pure maps, all $\mho^{pq}$ agree. We know that these are nontrivial $L_1(\Delta)$-centralizers at least when $\Omega$ is the Kalton-Peck map associated to an unbounded essentially concave Lipschitz function.

Note that all characters of $\Delta$ have order 2, so that one is now working on the assumption $(\ddagger)$ in the proof of Theorem~\ref{th:main}. Thus, taking $\Sigma$ as the sequence of Rademachers $(r_n)_{n\geq 1}$ we can consider the
Riesz products
$$
f= \prod_{k=1}^{N}\left( 1+\frac{i\, r_k}{\alpha\sqrt{N}}	\right),
$$
see (\ref{eq:rieszII}). The characters that can be obtained from $\{r_1,\dots, r_N\}$ are exactly the Walsh functions $w_a$ with $a\in\fin\{1,\dots, N\}$ and the length of $w_a$ agrees with $|a|$, the cardinality of $a$.
In this case we can be pretty concrete in the estimation of $\|\mho(f)\|_{L_1}$ because in view of Khintchine inequality the value of $C$ in (\ref{eq:esti}) is $\sqrt{2}=1/A_1$ and so $(\star)>0.3$ for $\alpha= 2$ so that 
\begin{equation}\label{eq:gthanlogn}
\|\mho(f)\|_{L_1}\geq 0.14\,\varphi(\log2\sqrt{N})-0.03\geq 0.07\, \log N
\end{equation}
for all $N$ if $\varphi$ is concave with $L_\varphi\leq 1$, in which case $Q(\mho^{\infty 1})\leq Q(\Omega)\leq 8/e$.
Actually for large values of $N$ we can do it much better because the sum $X_N=\sum_{j\leq N} r_j$ can be seen as a simple one-dimensional random walk, so the expected value  $\mathbb E[|X_N|]=\|X_N\|_{L_1}$ is exactly $2^{-N}\sum_{0\leq k\leq N}|2k-N|\binom{N}{k}$
and behaves asymptotically like $\sqrt{2N/\pi}$ -- cf. \cite[Chapter 6, p. 70]{circus}.

\subsection{Localizing the spatial support by means of the Fourier transform}

In the following result we partially localize the support of a given function on $\Delta$ by means of its Fourier coefficients.
Take $a\in\fin$, then $\e:a\To\{\pm 1\}$ and put $\Delta(a,\e)=\{t\in\Delta: t(k)=\e(k) \text{ for every } k\in a\}$.

\begin{lem}\label{lem:Dae}
For $f\in L_1(\Delta)$ TFAE:
\begin{enumerate}
\item $\supp(f)\subset \Delta(a,\e)$.
\item For every $k\in a$ one has $r_kf=\e(k) f$.
\item For every $k\in a$ and every $d\in D$ one has $\wh f(d\!\vartriangle\!\{k\})=\e(k) \wh f(d)$.
\end{enumerate}
\end{lem}

\begin{proof}
Obvious, after realizing that it suffices to check the equivalences for $|a|=1$.
\end{proof}

\begin{cor}
Let $\Omega: \ell_2^0(D)\To \ell_2(D)$ be a strictly symmetric quasilinear map and let $a, \e$ be as before. For every polynomial $f$ such that $\supp(f)\subset  \Delta(a,\e)$ one has
$\supp(\mho f)\subset  \Delta(a,\e)$. 
\end{cor}

\begin{proof}
This follows from the preceding lemma, taking into account that if $\Omega$ is strictly symmetric and if $g(a)=\e g(b)$ for $a,b\in \fin$ and $\e=\pm1$, then $(\Omega g)(a)=\e (\Omega g)(b)$.
\end{proof}

Each $a\in\fin$ induces a partition of $\Delta$ into $2^{|a|}$ clopen sets of equal measure, namely
\begin{equation}\label{eq:Pa}
\Delta= \bigsqcup_{\e\in\{\pm 1\}^a} \Delta(a,\e)
\end{equation}
If $X$ or $X(\Delta)$ is a space of functions on $\Delta$ and
 $S\subset \Delta$ is a clopen set let us write $X(S)_0=\{f\in X: f=0 \text{ off } S\}=\{f\in X: \supp f\subset S\}$. It is clear that each finite partition $\mathscr P$ of $\Delta$ into clopen sets induces decompositions
$$
L_p=\bigoplus_{S\in\mathscr P} L_p(S)_0\qquad\text{and} \qquad L_p^0=\bigoplus_{S\in\mathscr P} L_p^0(S)_0
$$
in which both direct sums carry the corresponding $\ell_p$-norm.

Now assume $\Phi:E\To F$ is a quasilinear map and that we have a decomposition $E=\bigoplus_{i\leq k} E_i$ so that each $x\in E$ can be written in a unique way as $x=\sum_{i\leq k} x_i$, with $x_i\in E_i$ and the ``projections'' $x\longmapsto x_i$ are bounded. 
Then, if we denote by $\Phi_i:E_i\To F$ the restriction of $\Phi$ to $E_i$ we can define a {\em new} quasilinear map 
$
\bigoplus_{i\leq k} \Phi_i : E=\bigoplus_{i\leq k} E_i\To F 
$
exactly as the reader is figuring out: $\bigoplus_{i\leq k} \Phi_i\big(\sum_{i\leq k}x_i\big)=\sum_{i\leq k}\Phi_i(x_i)$, where $x_i\in E_i$ for $i\leq k$. Using quasilinearity inductively it is very easy to see that $\|\Phi-  \bigoplus_{i\leq k} \Phi_i\|\leq M$, where the constant $M$ depends on $Q(\Phi)$ and the decomposition:
one can take $M\leq Q(\Phi)\lceil \log_2 k\rceil R$, where $\lceil\cdot\rceil$ is the {\em ceiling} function and $R=\sup_{\|x\|\leq 1}\sum_{i\leq k}\|x_i\|$. Furthermore, if $E$ is a $\mathscr K\!$-space one can replace $\lceil \log_2 k\rceil$ by the $\mathscr K\!$-space constant of $E$.

Such decompositions are quite useful in the study of our maps $\mho$ because these have many symmetries that are compatible with the decompositions induced by the partitions of the form (\ref{eq:Pa}), as we now see.

Given $a\in D$, consider two choices of signs $\e,\eta\in\{\pm 1\}^a$. 
If $y\in\Delta$ is defined by
\begin{equation}\label{eq:def-y}
y(n)=\begin{cases} 1 &\text{if $n\notin a$};\\
\e(n)\eta(n) &\text{if $n\in a$}	\end{cases}
\end{equation}
it is obvious that the (involutive) map $x\longmapsto xy^{-1}$ maps $\Delta(a,\e)$
to $\Delta(a,\eta)$ and vice-versa. Therefore,  $\supp f\subset \Delta(a,\e) \iff \supp f_y\subset \Delta(a,\eta)$ and $f\longmapsto f_y$ sends each function supported on $\Delta(a,\e)$ to one 
supported on $\Delta(a,\eta)$. This implies that if $\Omega$ is strictly symmetric then the ``chunks'' of $\mho$ associated to a partition of the form (\ref{eq:Pa}) are all basically ``equal''. Precisely:

\begin{lem} Assume $\Omega:\ell_2^0(D)\To \ell_2(D)$ is a strictly symmetric quasilinear map. Let $a\in D$ and $\e,\eta\in\{\pm1\}^a$ be as before and let $\mho_{a,\e}$ and $\mho_{a,\eta}$  be the restriction of $\mho$ to the polynomials supported on $\Delta(a,\e)$
and $\Delta(a,\eta)$, respectively.
 Then, if $y$ is given by $(\ref{eq:def-y})$ one has 
$(\mho_{a,\e}(f_y))_y= \mho_{a,\eta}(f)$ for every polynomial supported on $\Delta(a,\eta)$.
\end{lem}

Actually all the maps $ \mho_{a,\e}$ are ``copies'' of the mother map $\mho$ in the sense indicated in Proposition~\ref{prop:copies} below.
Indeed, given $a,\e$,  let $s:\N\To \N\backslash a $ be a bijection and set $(\sigma x)(n)=x(s(n))$. Given a function $f$ defined on $\Delta$ we define another function $E(f)$ {\em supported on}  $\Delta(a,\e)$
 letting
$$
(Ef)(x)=\begin{cases}
f(\sigma(x)) & \text{if $x\in \Delta(a,\e)$};\\
0 & \text{otherwise}.
\end{cases}
$$
Observe that $E$ depends not only on $a$ and $\e$ but also on $s$. It is clear that $E$ is multiplicative and that it defines isomorphisms between $L_p(\Delta)$ and $L_p(\Delta(a,\e))_0$ which preserve polynomials for all $p$. Actually $\|Ef\|_p=2^{-|a|/p}\|f\|_p$ for finite $p$ and $E:C(\Delta)\To C(\Delta(a,\e))_0$ is a surjective isometry. Let us denote by $E_p$ the map $E$ acting on $L_p(\Delta)$.

Here is what we wanted to establish:

\begin{prop}\label{prop:copies} Let $1\leq q\leq 2\leq p\leq\infty$.
If $\Omega: \ell_2^0(D)\To \ell_2$ is strictly symmetric then, for every $a$ and $\e$, one has $E_q\circ\mho^{pq}\approx \mho^{pq}_{a,\e}\circ E_p$ and, therefore, $  \mho_{a,\e}^{pq}\approx E_q\circ\mho^{pq} \circ E^{-1}_p  $.
\end{prop}

\begin{proof}[Sketch of the Proof]
Since $E$ commutes with the inclusions $L_r(\Delta)\To L_s(\Delta)$ for $s\leq r$ it suffices to do the proof for $p=q=2$. We need to describe the action of $E$ on Walsh functions. Since these are products of Rademachers and $E$ is multiplicative we just have to keep track of the Rademachers.
Given $x\in \Delta(a,\e)$ and $n\in\N$ one has
$$
(Er_n)(x) = r_n(\sigma(x))= (\sigma x)(n)=x(s(n)) \implies E(r_n)= 1_{\Delta(a,\e)}r_{s(n)}  \implies E(w_d)= 1_{\Delta(a,\e)} w_{s(d)},
$$
if we agree that $s(d)=\{s(n_1),\dots, s(n_k)\}$  for $d=\{n_1,\dots, n_k\}$.
The Fourier transform of $1_{\Delta(a,\e)}$ is supported on $\{b\in D; b\subset a\}$ and one actually has
$$
\wh 1_{\Delta(a,\e)}(b)=\begin{cases}
2^{-|a|}\prod_{n\in b}\e(n) & \text{if $b\subset a$},\\
0 & \text{otherwise},
\end{cases}
$$
although this sophistication is not necessary, so write just $1_{\Delta(a,\e)}=\sum_{b\subset a}c_b w_b$ and then
$$
E(w_d)= \left(\sum_{b\subset a}c_b w_b\right)w_{s(d)}= \sum_{b\subset a}c_b w_{b\vartriangle s(d)}= \sum_{b\subset a}E^b(w_d),
$$
where $E^b$ is the operator sending $w_d$ to $c_bw_{b\vartriangle s(d)}$.

Let us complete the proof assuming that $\mho$ commutes with every $E^b$; we will prove this right away -- see the next lemma.
Writing $E\mho(f)=\sum_{b\subset a} E^b\mho(f)=  \sum_{b\subset a} \mho(E^b f)$ we obtain, using that Hilbert spaces are $\mathscr K\!$-spaces,
\begin{align*}
\|\mho(Ef)-E\mho(f)\|_2&=
\left\|\mho\left(\sum\nolimits_{b} E^b f\right)- \sum\nolimits_{b} \mho(E^b f)\right\|_2\leq 37 Q(\mho) \sum\nolimits_{b}\| E^b f\|_2\\
&=37 Q(\Omega) \sum\nolimits_{b}2^{-|a|}\| f\|_2 = 37 Q(\Omega) \| f\|_2
\end{align*}
since the power set of $a$ has cardinality $2^{|a|}$.
\end{proof}

\begin{lem}
With the same notations as before, $\mho$ commutes with every $E^b$.
\end{lem}

\begin{proof}
It is clear that for each fixed $b\in\fin$ the map $\beta:\fin\To \fin$ defined by $\beta(d)=b\vartriangle s(d)$ is injective -- it is the composition of two injective maps. If $R^\beta$ is the corresponding rearrangement operator on $\ell_2(D)$ it is obvious that
$\mathscr F E^b= R^\beta \mathscr F$ and so $ E^b \mathscr F^{-1}= \mathscr F^{-1} R^\beta$. Hence, by Lemma~\ref{lem:abc},
$$
\mho(E^b f)= \mathscr F^{-1}\Omega(\wh{E^b f})=
\mathscr F^{-1}\Omega( R^\beta \wh f)= \mathscr F^{-1}  R^\beta \Omega( \wh f)= E^b \mathscr F^{-1} \Omega( \wh f) =E^b\mho (f).\qedhere
$$
\end{proof}

\subsection{An extension of $c_0$ by $\ell_1$ (explicit content)}

We now use our knowledge of the maps $\mho$ to construct a nontrivial quasilinear map $c_0\To \ell_1$. The existence of such objects is proved in \cite[Section 4.3]{2c-houston} (or \cite[Proposition~5.2.21]{2c}) by means of local arguments that, in the end, do not provide any explicit example and revisited in \cite[\S~3]{ccs} where it is shown that each nontrivial extension $0\To L_1(\mu)\To Z\To C(K)\To 0$ leads (by making pullback and pushout) to one of the form 
$0\To \ell_1\To Z'\To c_0\To 0$ which remains nontrivial. We first obtain the basic blocks of the construction. Given $n\in \mathbb N$ consider the group $\Delta_n=\{\pm1\}^n$ with Haar measure normalized to $1$. The dual group $D_n$ can be identified with the power set of $\{1,\dots,n\}$ with Haar measure normalized so that each point of $D_n$ has mass $1$.

Now let $\varphi:[0,\infty)\To [0,\infty)$ be an unbounded concave Lipschitz map vanishing at $0$, $\Omega_n: \ell_2(D_n) \To \ell_2(D_n)$ the corresponding Kalton-Peck map and $\mho_n^{\infty 1}: L_\infty(\Delta_n) \To L_1(\Delta_n)$ the map given by $\mho_n^{\infty 1}(f)= \mathscr F^{-1}\Omega_n(\wh f)$.
It is clear that the quasilinearity constants $Q(\mho_n^{\infty 1})$ are uniformly bounded (with a bound depending only on $L_\varphi$) and that there is a constant $c>0$ so that $\delta(\mho_n^{\infty 1})\geq c\varphi(\log n)$, in particular $\lim_n \delta(\mho_n^{\infty 1})=\infty$. Indeed, define a Riesz product with the $n$ Rademachers in $D_n$, use the estimate (\ref{eq:gthanlogn}) and then Corollary~\ref{cor:distvsbound}. Let $(c_k)$ be a summable sequence and, for each $k$, take $n(k)$ so large that $c_k\varphi(\log n(k))\To\infty$. The following result has an obvious proof.

\begin{lem}
The map $\Phi: c_0^0(\N, L_{\infty}(\Delta_{n(k)}))\To \ell_1(\N, L_1(\Delta_{n(k)}))$ defined by $\Phi((x_k)_{k\geq 1})=(c_k \mho_{n(k)}^{\infty 1}(x_k))_{k\geq 1}$ is quasilinear and nontrivial.
\end{lem}  

But $L_{\infty}(\Delta_{n(k)})= \ell_{\infty}(\Delta_{n(k)})$ with the same norm, while 
$L_{1}(\Delta_{n(k)})$ is isometric to $\ell_{1}(\Delta_{n(k)})$ so that $\Phi$ can be seen as a quasilinear map  $c_0^0\To \ell_1$ which, to be true, seems to be highly nonsymmetric and a bit nasty. Given the approach of this paper the following question is unavoidable, and rather intching:

\begin{prob}
{\rm Find a nontrivial quasilinear map $c_0^0(\Z)\To \ell_1(\Z)$ that commutes with translations or prove that there is none.}
\end{prob}

In a different direction we may ask if it is possible to construct a quasilinear map $\Phi:C(K)\To L_1(\mu)$ which is nontrivial when restricted to every subspace spanned by a disjoint sequence. Here $f$ and $g$ are {\em disjoint} if $fg=0$, equivalently, if $\supp f\cap \supp g=\varnothing$. This would lead to a strictly singular quasilinear map  $c_0\To \ell_1$ whose existence, asked repeatedly by J.M.F. Castillo, we do not known either. See \cite{marilda} for the available information on twisted sums of $\ell_1$ and $c_0$ and observe that our results do not solve any of the problems stated or suggested at the end of that paper.

\section{Miscellaneous remarks}

\subsection{Is there a nonlinear Grothendieck theorem?}~\label{sec:CT}
The version of Grothendieck theorem that best suits our purposes is \cite[Theorem 2.1]{pisier}: Let $S,T$ be compact spaces and $	\mho: C(S)\To M(T)$ an operator. Then there are probabilities $\lambda, \mu$ respectively on $S$ and $T$ and an operator $\Omega: L_2(\lambda)\To  L_2(\mu)$ such that $\mho=J\Omega I$, where $I:C(S)\To L^2(\lambda) $ and $J :L^2(\mu) \To L^1(\mu)\To M(T) $ are the formal inclusions. We have slightly edited Pisier statement  to emphasize that all known quasilinear maps $	\mho: C(S)\To M(T)$ have a similar factorization where $\Omega: L_2(\lambda)\To  L_2(\mu)$ is now quasilinear. We do not know if this is due to a general principle or rather to our lack of ability to find other examples.

\subsection{The failure of Kalton-Roberts theorem for bimeasures}
Kalton-Roberts' proof that Banach spaces of type $\mathscr L_\infty$ are $\mathscr K\!$-spaces basically consists in showing that there is a constant $K$ so that whenever $m: 2^S \To \R$ satisfies $|m(A\cup B)-m(A)-m(B)|\leq 1$ for all disjoint subsets $A,B$ of a finite set $S$  there is a true measure $\mu: 2^S \To \R$ such that $|m(A)-\mu(A)|\leq K$ for all $A\subset S$.
One may wonder if bimeasures have a similar stability property, namely if given a function $b: 2^S\times 2^T \To \mathbb R$ that satisfies the estimate $|b(A\cup B, C)-b(A, C)-b(B, C)|\leq 1$ for all disjoint $A,B\subset S$ and all $C\subset T$ and the same with the variables reversed, there is {\em bimeasure}  $\beta: 2^S\times 2^T \To \mathbb R$ such that $|b(A,B)-\beta(A,B)|\leq K$ for some {\em universal} constant $K$.
The existence of quasilinear maps $\Phi_n:\ell_\infty^n\To \ell_1^n$ with $Q(\Phi_n)$ uniformly bounded and $\delta(\Phi_n)\To\infty$ implies that there is no Kalton-Roberts theorem for bimeasures.

Our results can only give $\delta(\Phi_n)/Q(\Phi_n)\geq c\log\log(n)$ for some $c>0$. Brudnyi and Kalton have already obtained $\delta(\Phi_n)/Q(\Phi_n)\geq c\log(n)$ in \cite{b-k}.

\subsection{A more professional approach}
Knowledgeable people on centralizers use a more flexible definition that requires to have a good ambient space. Precisely, let $X, Y$ be Banach $A$-modules and let $W$ be another $A$-module containing $Y$ in the purely algebraic sense. A homogeneous mapping $\Phi: X\To W$ (not $Y$) is said to be a centralizer from $X$ to $Y$ if for every $x,y\in X$ and every $a\in A$ the differences $\Phi(x+y)- \Phi(x)-\Phi(y)$ and $\Phi(ax)-a\Phi(x)$ fall in $Y$ and obey the estimates $\|  \Phi(x+y)- \Phi(x)-\Phi(y)\|\leq Q(\|x\|+\|y\|)$ and $\|\Phi(ax)- a\Phi(x)\|\leq C\|a\|\|x\|$ for some constants $Q$ and $C$.

The set $Y\oplus_\Phi X=\{(w,x)\in W\times X: w-\Phi(x)\in Y\}$ is a submodule of $W\times X$ (with the coordinatewise product) which becomes a quasinormed module with the quasinorm $\|(w,x)\|_\Phi=\|w-\Phi(x)\|+\|x\|$. The sequence 
\begin{equation}\label{eq:mumu}
\xymatrix{
0\ar[r] & Y \ar[r]^-\imath & Y\oplus_\Phi X \ar[r]^-\pi & X\ar[r] &0
}
\end{equation}
in which $\imath(y)=(y,0), \pi(w,x)=x$ is exact. Besides, $\imath$ preserves the quasinorm and $\pi$ maps the unit ball of $Y\oplus_\Phi X$ onto that of $X$. This implies that $Y\oplus_\Phi X$ is complete (hence a quasi Banach $A$-module) if $X$ and $Y$ are and thus no further action is required. The splitting criterion for these sequences is simple: (\ref{eq:mumu}) splits, as an extension of quasi Banach modules, if and only if there is a morphism $\phi:X\To W$ such that $\Phi-\phi$ is bounded from $X$ to $Y$. 

There is a classical object that could have been used as an ambient space for $L_1$-centralizers: the space of {\em pseudomeasures} on $G$.

Consider the Fourier algebra $A(G)=\{f\in C(G): \sum_{\gamma\in\Gamma}|\wh f(\gamma)|<\infty\}$ and norm it by $\|f\|_{A(G)}= \sum_{\gamma\in\Gamma}|\wh f(\gamma)|$ so that $A(G)$ is isometric to $\ell_1(\Gamma)$.  Clearly, $A(G)$ is an $M(G)$-module under convolution (the Fourier coefficients of any finite measure are bounded). Let $A^*(G)$ be the dual space of $A(G)$, which is an $M(G)$-module under the ``dual action'':
$$
\langle \mu\circledast\phi, f\rangle= \langle \phi, f*\mu \rangle\qquad(\mu\in M(G),\phi\in A^*(G), f\in A(G)).
$$
Note that, as $A(G)$ is (uniformly) dense in $C(G)$, every measure can be  treated as an element of $A^*(G)$ and the ``inclusion'' is continuous. Besides, $\mu\circledast\phi= \mu*\phi$ if $\phi\in M(G)$.  Of course there is nothing really {\em new} in the space $A^*(G)$ which is just $\ell_\infty(\Gamma)$ in disguise.
What is interesting for us is that we can interpret $\mathscr F:A^*(G)\To \ell_\infty(\Gamma)$ as an isomorphism in the obvious way and then consider $A^*(G)$ as a ``convolution'' algebra defining $\phi\circledast \psi= \mathscr{F}^{-1}\big(\mathscr F(\phi) \mathscr F(\psi))$. 

We have the following companion of Lemma~\ref{lem:all}:

\begin{prop} Let $X, Y$ be Banach $L_1$-modules.
If $Y$ admits a continuous inclusion into  $A^*(G)$ then every extension  of $X$ by $Y$ (in the category of Banach $L_1$-modules) is equivalent to the extension induced by an $L_1$-centralizer from $X$ to $Y$, with values in  $A^*(G)$.
\end{prop}

\begin{proof}[Sketch of the Proof]
Let $\xymatrix{
0\ar[r] & Y \ar[r]^-\jmath & Z \ar[r]^-\varpi & X\ar[r] &0
}
$
be an extension and consider the pushout diagram (see Section~\ref{sec:POPB})
 $$\xymatrix{
0\ar[r] & Y \ar[d]^\jmath \ar[r]^-\imath & Z  \ar[d] \ar[r]^-\pi & X\ar[r]\ar@{=}[d] &0\\
0\ar[r] & A^*(G) \ar[r] & \PO \ar[r]^-\varpi & X\ar[r] &0
}
$$
in which $\jmath$ is the hypothesized inclusion -- an injective homomorphism. Since $A(G)^*\approx \ell_\infty(\Gamma)$ is an injective Banach space the lower row splits in the linear category and the amenability of $L_1(G)$-implies that the same happens in the category of modules. Thus, there is a homomorphism $J:Z\To  A^*(G)$ extending the inclusion of $Y$ into  $A^*(G)$ (that is, $\imath J=\jmath$). Let $B:X\To Z$ be a bounded homogeneous section for the quotient map $\pi$ (that is, $\pi B={\bf I}_X$). Then the composition
$$
\xymatrixcolsep{4.5pc}
\xymatrix{
\Phi: X\ar[r]^-B_-{\text{nonlinear}} &Z \ar[r]^-J_-{\text{homomorphism}}  & A^*(G)
}
$$
is an $L_1$-centralizer from $X$ to $Y$ (check it!) and induces an extension equivalent to the starting one.
\end{proof}

\subsection{Group actions and $G$-modules}
As we already mentioned, module structures on convolution algebras are often related to actions of the underlying group. Given a topological group $G$, a (quasi) Banach $G$-module is a (quasi) Banach space $X$ together with a continuous {\em action} $G\times X\To X$, that is, a continuous representation of $G$ into the algebra of operators on $X$. When $G$ is a compact abelian group the notions of $G$-module and that of $L_1(G)$-module are more or less interchangeable for Banach spaces of functions on  $G$. Things are different, however, for nonlocally convex quasi Banach spaces and indeed the space  $L_p(G)$ for $0<p<1$ is the 
perfect example of a $G$-module whose action cannot be ``extended'' to a module structure over the convolution algebra $L_1(G)$. Our methods, based on the Fourier transform, do not apply in general to nonlocally convex spaces, the only exception being Proposition~\ref{prop:LooL1} that would require some additional work. We refer the reader to \cite{c-f,kuch} for some results on exact sequences of $G$-modules.

\subsection*{Acknowledgement} We thank \'Oscar Blasco for his explanations and patience.

\end{document}